\documentclass[a4paper]{article}

\usepackage{amsfonts, amsmath, amsthm, empheq, amssymb, color, indentfirst}

\allowdisplaybreaks[4]

\newtheorem{thm}{Theorem}

\newtheorem{lem}{Lemma}[section]
\newtheorem{rem}{Remark}[section]
\newtheorem{prob}{Problem}

\renewenvironment{abstract}{%
        \small
        \quotation
         \noindent {\bfseries \abstractname } }%
      {\if@twocolumn\else\endquotation\fi}

\def\R{\mathbb R}
\def\N{\mathbb N}
\def\Z{\mathbb Z}
\def\Q{\mathbb Q}
\def\De{\Delta}
\def\al{\alpha}
\def\be{\beta}
\def\de{\delta}
\def\ep{\epsilon}
\def\sg{\sigma}
\def\Sg{\Sigma}
\def\d{\mathrm d}
\def\e{\mathrm e}
\def\f{\frac}

\def\ga{\gamma}
\def\Ga{\Gamma}
\def\ve{\varepsilon}
\def\te{\theta}
\def\la{\lambda}
\def\vp{\varphi}
\def\na{\nabla}

\def\Om{\Omega}
\def\ov{\overline}
\def\pa{\partial}

\def\wt{\widetilde}

\def\cC{\mathcal{C}}
\def\cD{\mathcal{D}}

\def\dx{\d x}
\def\dt{\d t}
\def\LL{{L^2(\Om)}}

\title{\Large\bf
Carleman estimates for the time-fractional advection-diffusion equations and applications}

\author{\large Zhiyuan LI$^\dag$ \qquad Xinchi HUANG$^\dag$ \qquad Masahiro YAMAMOTO$^\dag$}

\date{}

\begin{document}
\maketitle

\renewcommand{\thefootnote}{\fnsymbol{footnote}}
\footnotetext{\hspace*{-5mm} 
\begin{tabular}{@{}r@{}p{16cm}@{}} 
& Manuscript last updated: \today.\\
$^\dag$ 
& Graduate School of Mathematical Sciences, 
the University of Tokyo,
3-8-1 Komaba, Meguro-ku, Tokyo 153-8914, Japan. 
E-mail: 
zyli@ms.u-tokyo.ac.jp,
huangxc@ms.u-tokyo.ac.jp,
myama@ms.u-tokyo.ac.jp.
\end{tabular}}

\begin{abstract}
In this article, we prove Carleman estimates for the generalized time-fractional advection-diffusion equations by considering the fractional derivative as 
perturbation for the first order time-derivative.
As a direct application of the Carleman estimates, we show a conditional stability of a lateral Cauchy problem for the time-fractional advection-diffusion equation, and we also investigate the stability of an inverse source problem.

\vskip 4.5mm

\noindent\begin{tabular}{@{}l@{ }p{9.5cm}} {\bf Keywords } &
time-fractional advection-diffusion equation,
Carleman estimate,
lateral Cauchy problem,
inverse source problem
\end{tabular}

\vskip 4.5mm


\end{abstract}

\section{Introduction and main results}
\label{sec-intro}

Recently, the position of the advection-diffusion equations (ADE) as models in a range of problems in analyzing 
mass transport has been challenged by more and more experiment data. For example, numerous field 
experiments for the solute transport in highly heterogeneous media demonstrate that solute concentration 
profiles exhibited anomalous non-Fickian growth rates, skewness and long-tailed profile (See e.g., 
\cite{BWM00} and \cite{LB03}), which are poorly characterized by the conventional mass transport equations 
based on Fick's law. To more accurately interpret these effects, the non-Fickian diffusion model has been 
proposed to mass transport model, say, time-fractional diffusion equation (FDE):
\begin{equation}
\label{equ-fade}
\pa_t u(x,t) + {\ga}\pa_t^\al u(x,t) =\De u(x,t),\quad (x,t)\in \R^n \times(0,\infty),
\end{equation}
where for $\al\in(0,1)$, by $\pa_t^\al$ we denote the Caputo derivative with respect to temporal variable $t>0$:
$$
\pa_t^\al g(t) = \f1{\Ga(1-\al)} 
\int_0^t (t-\tau)^{-\al}\f{\d g(\tau)}{\d\tau} \d\tau,
$$
where $\Ga(\cdot)$ is the usual Gamma function. See, e.g., \cite{KST06} and \cite{P99} for the properties of the Caputo derivative, and see, e.g., \cite{HWJ02}, \cite{SB03} and the references therein for the FDEs. 

Introducing the time-fractional derivatives of arbitrary order into the equation of mass transport
for a heterogeneous medium achieved great successes, for example, it is shown to be an efficient model for 
describing some anomalous diffusion processes in the highly heterogeneous media by \cite{HH98} in which the 
authors pointed out that diffusion equation with time-fractional derivative was well-performed in describing the 
long-tailed profile of a particle diffuses in a highly heterogeneous medium, and by \cite{MMP} where the 
theoretical fractional calculus on FDEs shows that there holds the non-Fick's law in the anomalous diffusion. We 
also refer to \cite{SB03} in which the MADE site mobile tritium mass decline is well modeled by {the equation \eqref{equ-fade} with the} time-fractional derivative of order $\al=0.33$.

In this paper, assuming $0<\al_\ell<\cdots<\al_1<1$, we consider a generalized time-fractional advection-diffusion equation (FADE) 
\begin{equation}
\label{equ-gov}
(Lu)(x,t)\equiv \pa_t u +\sum_{j=1}^\ell q_j(x,t) \pa_t^{\al_j} u 
-\sum_{i,j=1}^n {a_{ij}(x,t)} \pa_i\pa_ju
{+}\sum_{i=1}^n b_i(x,t)\pa_iu {+} c(x,t)u=F,
\end{equation}
where $(x,t)$ in $\R^n \times(0,\infty)$.
The above equation \eqref{equ-gov} has the spatially and temporally variable coefficients, and such kind of 
equations simulate the advection diffusion, which is more general than that in \cite{HWJ02} and \cite{SB03}, and 
so can be regarded as more feasible model equation than symmetric fractional diffusion equations in modeling 
diffusion in heterogeneous media. Here in this article, we study the stability for a lateral Cauchy problem and 
the stability for an inverse source problem for this equation, and the stability is a fundamental theoretical 
subject. To the best knowledge of the authors, except the special case that $\alpha=1/2$ which is discussed in \cite{KY16}, the stability results for both of the {lateral} Cauchy problem and 
inverse source problem of the equation \eqref{equ-gov} were not yet established. Here it should be mentioned that in the general order case, the transform argument and Fourier methods used in the above mentioned references \cite{HWJ02} and \cite{SB03} can not work any more because of the non-symmetry of the system and $t$-dependent coefficients, and it is very complicated to follow the treatment used in \cite{KY16} even for $\alpha=1/3$. One of the reasons is that in this case one needs to establish a Carleman estimate for parabolic operators of order $6$ in the space variables, which will cause a huge amount of computation. However, from the shape of 
the equation \eqref{equ-gov} we regard the lower fractional order term as a perturbation of the first order 
time-derivative, which enable one to derive the Carleman estimate for the equation \eqref{equ-gov} in the 
framework of the Carleman estimate for the parabolic equations, and then consider the stability of the lateral 
Cauchy problem and inverse source problem. 

To this end, we start from fixing some general settings and notations. Let $T>0$ be fixed constant and 
$\Om\subset\R^n$ is a bounded domain, $n\ge1$, with sufficiently smooth boundary $\pa\Om$, for example, of
$C^2$-class. We set $Q:=\Om\times(0,T)$. Assume that $a_{ij}=a_{ji}\in C^1(\ov Q)$, $1\le i,j\le n$, satisfies
that 
$$
\rho \sum_{j=1}^n \xi_j^2 \le \sum_{j,k=1}^n a_{jk}(x,t) \xi_j\xi_k, \quad (x,t)\in\ov{Q},\ \xi\in\R^n,
$$
where $\rho>0$ is a constant independent of $x,t,\xi$. We set 
$\pa_{\nu_A} u = \sum_{i,j=1}^n a_{ij}\nu_i \pa_j u$ where $(\nu_1,\cdots,\nu_n)$ denotes the unit outwards 
normal vector to the boundary {$\pa\Om$}. Let $\LL$ and $H^{k,\ell}(Q)$ ($k\ge0,\ell\ge0$) denote 
Sobolev spaces (See, e.g., \cite{A75} and \cite{Y09}). Similar to Theorem 5.1 in \cite{Y09}, for arbitrary non-empty relatively open sub-boundary $\Ga\subset\pa\Om$, we choose a 
bounded domain $\Om_1$ with sufficiently smooth boundary such that
\begin{equation}
\label{con:Om1}
\Om {  \subsetneq} \Om_1, \quad \ov{\Ga} = \ov{\pa\Om\cap\Om_1}, \quad 
\pa\Om\setminus\Ga \subset \pa\Om_1.
\end{equation}
We then apply Lemma 4.1 in \cite{Y09} to {find a function} $d\in C^2(\ov{\Om_1})$ satisfying
\begin{equation}
\label{con:d}
d>0 \mbox{ in }\Om, \quad |\na d| > 0 \mbox{ on } \ov\Om, \quad d=0 \mbox{ on }\pa\Om\setminus\Ga.
\end{equation}

Now let us first consider the equation \eqref{equ-gov} with $\al_1<\f12$ (we call the corresponding diffusion as 
sub-diffusion), we set the weight function {$\vp_1$} as follows
$$
\vp_1(x,t)=\e^{\la\psi_1(x,t)},\quad \psi_1{  (x,t)}=d(x) - \be t^{2-2\al_1}, \quad \forall\la\ge0,\ {  x\in\ov{\Om_1},\ t\ge 0},
$$
where $\be>0$ is a positive constant which will be chosen later, and we have the following Carleman type 
estimate for the equation \eqref{equ-gov}
\begin{thm}
\label{thm-CE<0.5}
We assume $\al_1<\f12$ and $q_i,b_j,c\in L^\infty(Q)$ $(i=1,\cdots,\ell, j=1,\cdots,n)$ in the equation 
\eqref{equ-gov}, and we let $\Sg_0=\ov\Om\times\{0\}$ and $D\subset Q$ be bounded domain whose 
boundary $\pa D$ is composed of a finite number of smooth surfaces. Then there exists a constant $\la_0>0$ 
such that for arbitrary $\la\ge\la_0$, we can choose a constant $s_0(\la)>0$ satisfying: there exist constants 
$C=C(s_0,\la_0)>0$ and $C(\la)>0$ such that
\begin{align}
\label{esti-CE<0.5}
&\int_D\left\{\f{1}{s\vp_1}|\pa_t u|^2+s\la^2\vp_1|\na u|^2
+ s^3\la^4\vp_1^3 u^2 \right\} \e^{2s\vp_1} \dx\dt
\nonumber\\
\le& C\int_D |\wt Lu|^2 \e^{2s\vp_1} \dx\dt
+{C(\la)}\e^{C(\la)s}\int_{\pa D} (|\na u|^2 + u^2) {  \d \Sg}
+ {C(\la)}\e^{C(\la)s}\int_{\pa D\setminus\Sigma_0} |\pa_t u|^2 \d\Sg
\end{align}
for all {$s\ge s_0$} and all $u\in H^{2,1}(D)$, where $\wt L:= L - \sum_{j=1}^\ell q_j(x,t)\pa_t^{\al_j}$.
\end{thm}
From the above Carleman estimate, similar to the argument used in \cite{Y09}, we further have the conditional stability for the lateral Cauchy problem for the equation \eqref{equ-gov}. However, a bit different from the results in \cite{Y09}, here due to the choice of the weight function in the derivation of the Carleman estimate in Theorem \ref{thm-CE<0.5}, we can only prove that the continuous dependency of the solution with respect to initial values, boundary values and source terms in the case of {$\al_1\in(0,\f12)$}, say,
\begin{thm}
\label{thm-LCP<0.5}
We assume $\al_1<\f12$ and $q_i,b_j,c\in L^\infty(Q)$ $(i=1,\cdots,\ell, j=1,\cdots,n)$ in the equation 
\eqref{equ-gov}. Let $\Ga\subset\pa\Om$ be an arbitrary non-empty relatively open sub-boundary of $\pa\Om$. For any 
bounded domain $\Om_0$ such that $\ov{\Om_0}\subset\Om\cup\Ga$, 
$\pa\Om_0\cap\pa\Om \subsetneq \Ga$ is a non-empty open subset of $\pa\Om$, we can choose a 
sufficiently small $\ve=\ve(T,\Om_0)>0$ such that 
\begin{align}
\label{esti-Cauchy<0.5}
\|u\|_{H^{1,1}(\Om_0\times(0,\ve))}
\le C\|u \|_{H^{1,1}(Q)}^{1-\te} \cD^\te,
\end{align}
where $\cD:=\|u(\cdot,0)\|_{H^1(\Om)}+\|F\|_{L^2(Q)} + \|u\|_{H^1(\Ga\times(0,T))} +
\|\pa_{\nu_A} u\|_{L^2(\Ga\times(0,T))}$, and the constants $C>0$ and $\te\in(0,1)$ may depend on $T$, the choice of $\Om_0$ and the coefficients of the equation \eqref{equ-gov}.
\end{thm}

The above arguments used for deriving Theorem \ref{thm-LCP<0.5} cannot work anymore for the general 
fractional order $\al_1\ge\f12$ (the corresponding diffusion is called as sup-diffusion), but for some fractional orders, it is expected that the method of Carleman estimate still works. As a partially affirmative answer, we focus on deriving the Carleman estimate for the case of rational fractional order 
$\al_1$ which is smaller than $\f34$, i.e. $\al_1\le \f34$ in the form of $\f{m}k$, $m,k\in\N$. For this, we consider the following equation
\begin{equation}
\label{equ-gov2}
(Lu)(x,t)\equiv\pa_t u+q(x)\pa_t^\al u-\sum_{i,j=1}^n a_{ij}(x) \pa_i\pa_ju {+}\sum_{i=1}^n b_i(x)\pa_iu{+}c(x)u=F
\end{equation}
in the case of $\al\in (0,\f34]$, where the coefficients satisfy: {$b_i\in L^\infty(\Om)$, $i=1,\cdots,n$, 
$c\in L^\infty(\Om)$} and the source term $F$ is assumed to be smooth enough. To this end, we first introduce 
the Riemann-Liouville fractional derivative of order {$\al\in[m-1,m)$ with $m=1,2,\cdots$ which is usually defined by
\begin{equation}
\label{def-RLD}
D_t^\al h(t):=\f1{\Ga(m-\al)} \f{\d^m}{\dt^m}\int_0^t (t-r)^{-\al+m-1} h(r) \d r
\end{equation}
}
{
and the Riemann-Liouville fractional integral $D_t^{-\al}$ of order $\al\in(0,1)$ which is defined by
\begin{equation}
\label{def-RLI}
D_t^{-\al} h(t):=\f1{\Ga(\al)} \int_0^t (t-r)^{\al-1} h(r) \d r.
\end{equation}
}
We choose the weight function $\vp_2$: 
\begin{equation}
\label{con:weight}
\vp_2(x,t) = e^{\la\psi_2(x,t)}, \quad \psi_2(x,t) = d(x) - \beta(t - t_0)^2 + c_0,
\end{equation}
where $t_0\in (0,T)$ and $\beta>0$ are to be fixed later and $c_0 := \max\{\beta t_0^2, \beta (T - t_0)^2\}$ guarantees the non-negativeness of the function $\psi_2$.
In needs of applications, we establish a special Carleman estimate with a cut-off function 
$\chi_0\in C^\infty(\R^{n+1})$ which is defined by
$$
{\chi_0}(x,t)=
\begin{cases}
1, &(x,t)\in D_0,\\
0, &\mbox{outside of } D,
\end{cases}
$$
where $D$ is an arbitrary subdomain of $Q_1:=\Om_1\times(0,T)$ and $D_0\subsetneq D$. Now we are ready to state the Carleman estimate.
\begin{thm}
\label{thm-CE<=0.75}
For any rational number $\al= \f{m}{k}\le \f34$, $m,k\in\N$, we assume that $D_t^{\f{j}k} F \in L^2(Q)$ for $j=j_1,...,j_k$ 
where $j_l := -\f{k}{2} + \f{(-1)^k -1}{4} + l, l=1,...,k$.
Then there exist constants $\hat{s}\ge 1$ and $C>0$ such that
\begin{align*}
&\int_Q \chi_0^2\left(\sum_{i,l=1}^n (s\vp_2)^{\f{4}{k}j_1 -1}|\pa_i\pa_l u|^2 + (s\vp_2)^{\f{4}{k}j_1 +1}|\na u|^2 + \sum_{j=j_1}^{j_k+k} (s\vp_2)^{-\f{4}{k}(j-j_1)+3} |D_t^{\f{j}k} u|^2\right)\e^{2s\vp_2} \dx\dt \\
&\hspace{0cm}\le C\int_Q \chi_0^2\left(\sum_{j=j_1}^{j_k} (s\vp_2)^{-\f{4}{k}(j-j_1)} |D_t^{\f{j}k} F|^2\right)\e^{2s\vp_2} \dx\dt + Low + Bdy
\end{align*}
for all $s\ge \hat{s}$, large fixed $\la \ge 1$ and all $u$ smooth enough satisfying \eqref{equ-gov2}  and  $u(x,0)=0$ for $\forall x\in\Om$, where $Low$ and $Bdy$ are defined as follows
\begin{align}
\label{Low}
&Low 
:= Cs\int_Q \left(|\pa_t\chi_0|^2 + |\na\chi_0|^2 + \sum_{i,j=1}^n |\pa_i\pa_j \chi_0|^2\right)\left(\sum_{j=j_1}^{j_k} \left(|\na(D_t^{\f{j}k} u)|^2 + |D_t^{\f{j}k} u|^2 \right)\right)\e^{2s\vp_2} \dx\dt, \\
\label{Bdy}
&Bdy 
:= C\e^{Cs}\int_{\pa Q} \left(|\pa_t\chi_0|^2 + |\na\chi_0|^2 + |\chi_0|^2\right) \left(\sum_{j=j_1}^{j_k} \left(|\na_{x,t} (D_t^{\f{j}{k}} u)|^2 + |D_t^{\f{j}{k}} u|^2 \right)\right) \d S\dt.
\end{align}
\end{thm}
We notice that the above Carleman estimate with the regular weight function $\vp_2$ requires the 
homogeneous initial condition, which possibly roots in the memory effect of time-fractional {derivatives}. 
We refer to the publications \cite{CLN09} and \cite{XCY} for the Carleman estimate for the diffusion equations with 
a half order time-fractional derivative where the homogeneous initial value is also necessary for deriving the 
Carleman estimates. Moreover, it should be mentioned here that our method can work for multi-term case where the fractional order $\al_j$, $j=1,...,l$ are all rational numbers and the largest one $\al_1$ is smaller than $\f34$.

As an application of Theorem \ref{thm-CE<=0.75}, we can easily derive the stability for the lateral Cauchy problem for the equation \eqref{equ-gov2}. For simplicity, we only state the theorem in the critical case: $\al=\f34$. The results for the case of rational order less than $\f34$ can be established by using the similar argument.

\begin{thm}
\label{thm-LCP=0.75}
For any small $\ep>0$ and any bounded domain $\Om_0$ satisfying $\ov{\Om_0}\subset\Om\cup\Ga$, $\pa\Om_0\cap\pa\Om$ be a non-empty open subset of $\pa\Om$ and $\pa\Om_0\cap\pa\Om\subsetneq\Ga$, there exist constants $C>0$ and $\te\in (0,1)$ such that
$$
\|u\|_{H^{2,1}(\Om_0\times(\ep,T-\ep))} \le C \cD + CM^{1-\te}\cD^\te
$$
for $u$ satisfying equation \eqref{equ-gov2} with $\al = \f34$ and $u(x,0)=0$ for $\forall x\in\Om$. Here $M$ and $\cD$ are defined as follows
{
\begin{align*}
&M :=\|D_t^{\f12} u\|_{H^{1,0}(Q)}, \\
&\cD := \sum_{j=-1}^2 \|D_t^{\f{j}4} F\|_{L^2(Q)} + \|D_t^{\f12} u\|_{H^1(\Ga\times(0,T))} + \|{\pa_\nu} (D_t^{\f12} u)\|_{L^2(\Ga\times(0,T))}.
\end{align*}
}
\end{thm}
\begin{rem}
Here the generic constant $C$ and constant $\te$ depend on the choice of $\Om_0$, $\ep$ and the 
coefficients of the equation \eqref{equ-gov2}.
\end{rem}

Now on the basis of the above Carleman estimate, let us turn to considering another application: inverse source problem for the equation \eqref{equ-gov2} where the source term is in the form of 
$F(x,t) = R(x,t)f(x)$. Here again we consider the critical case:
\begin{equation}
\label{sy:ISP}
{(Lu)(x,t)} = \pa_t u + { q(x)\pa_t^{\f34}} u - \Delta u + B(x)\cdot\na u + {c(x)} u = R(x,t)f(x) \quad \mbox{in $Q$,}
\end{equation}
where $B(x):=(b_1(x), \cdots, b_n(x))$.
\begin{prob}[Inverse source problem]
\label{prob-ISP=0.75}
Fix an observation time $t_0\in(0,T)$. We intend to determine the spatially varying factor $f$ for given $R$ by 
measuring the data on some sub-boundary and the value of the solution $u$ at $t=t_0$.
\end{prob}

The measurements are as the same type as that in the case of a heat equation (See, e.g., \cite{Y09} for a 
similar inverse source problem to heat equation). In our problem, we deal with a parabolic equation with some
lower-order time-fractional derivative. The idea is to put the fractional derivative term into source term and 
give some suitable estimates. We have the following conditional stability result on a level set $\Om_{\ep} := \{x\in\Om: d(x) > \ep\}$ for any $\ep>0$.
\begin{thm}
\label{thm-ISP=0.75}
Assume that $R\in L^\infty(Q)$ with condition $R(\cdot,0) = 0$ in $\Om$ and $R$ satisfies
\begin{equation}
\label{con:R}
R(\cdot,t_0)\ne0 \quad on\ \ov{\Om},\qquad D_t^{\f12}R\in L^\infty(Q).
\end{equation}
Then {for any $\ep>0$} there exist constants $C>0$ and $\te\in (0,1)$ such that
$$
\|f\|_{L^2(\Om_{4\ep})} \le C\cD + CM^{1-\te}\cD^\te
$$ 
for all $u$ smooth enough and satisfying the equation \eqref{sy:ISP} with
\begin{equation}
\label{con:initial}
u(\cdot,0) = \pa_t u(\cdot,0) = 0 \quad in \ \Om.
\end{equation}
Here by $M$ and $\cD$ we denote a priori bound and measurements as follows
\begin{align*}
& M := \|f\|_{L^2(\Om)} + \left\|D_t^{\f32} u\right\|_{H^{1,0}(Q)}, \\ 
& \cD := \|u(x,t_0)\|_{H^2(\Om)} + \left\|D_t^{\f32} u\right\|_{H^1(\Ga\times (0,T))} 
+ \left\|D_t^{\f32} (\pa_{\nu} u)\right\|_{L^2(\Ga\times (0,T))},
\end{align*}
where the generic constant $C$ and constant $\te$ depend on $\ep$ and the coefficients in the equation \eqref{sy:ISP}.
\end{thm}
To the best knowledge of the authors, most of the existing literatures are focused on the uniqueness of the 
inverse problems for the time-fractional diffusion equation, see, e.g., \cite{CNYY}, \cite{JLLY16}, \cite{LIY15}, 
\cite{LY15}, \cite{Z16} and the references therein. This is a first attempt to attack the stability of the inverse 
source problem (Problem \ref{prob-ISP=0.75}) for the time-fractional advection-diffusion equation. Moreover, due to our 
methods, it is necessary to assume both solution and the time derivative of the solution vanish at the initial time 
although the Carleman estimate Theorem \ref{thm-CE<=0.75} used for deriving Theorem \ref{thm-ISP=0.75} holds true only provided the homogeneous initial value.

The rest of this paper is organized in three sections. 
In Section \ref{sec-order<0.5},  by regarding the fractional-order terms as non-homogeneous terms and applying 
the Carleman estimate for the parabolic equations, we will give a proof for Theorem \ref{thm-CE<0.5} in the 
case of the highest fractional order is strictly less than half, and then as a direct application, the conditional 
stability of a lateral Cauchy problem for the equation \eqref{equ-gov} stated in Theorem \ref{thm-LCP<0.5} will 
be established. 
In Section \ref{sec-order<=0.75}, we first finish the proof for Theorem \ref{thm-CE<=0.75} with a regular weight 
function which is usually used dealing with the problems in the parabolic equations, and on the basis of the 
Carleman type estimate in Theorem \ref{thm-CE<=0.75}, we will show that the solution continuously depends on 
Cauchy data and source term. Finally, concluding remarks are given in Section \ref{sec-rem}. 

\section{Carleman estimate for the sub-diffusion and its applications}
\label{sec-order<0.5}
In this section, we investigate the equation \eqref{equ-gov} with fractional order $\al_1<\f12$. We point out that 
the equation \eqref{equ-gov} can be regarded a parabolic type equation if we regard the lower fractional order 
terms as new {non-homogeneous} terms. Therefore, it is expected to employ the Carleman estimate for the 
parabolic equations to derive the Carleman estimate for our equation, which is the key idea in this section.
Owing to this treatment, in Section \ref{sec-CE<0.5}, we will give the proof of Theorem \ref{thm-CE<0.5}, while 
Theorem \ref{thm-LCP<0.5} will be proved as an application in Section \ref{sec-LCP<0.5}.

\subsection{Carleman estimate for the sub-diffusion}
\label{sec-CE<0.5}
In this subsection, recalling the notations $d\in C^2(\ov\Om)$ and $|\na d|\ne0$ on $\ov\Om$, and the function 
$\psi_1=d(x)-\be t^{2-2\al_1}$ with $\be>0$, we follow the arguments on pp.9-19 of the survey paper 
\cite{Y09} to prove the Carleman estimate \eqref{esti-CE<0.5}. We use the same notations, where we must 
modify locally because our choice of the time dependency of $\psi_1$ is different.

\begin{proof}[\bf Proof of Theorem \ref{thm-CE<0.5}]
It is sufficient for us to discuss the derivation of a Carleman estimate for 
$L_0 = \pa_t-\sum_{i,j=1}^n a_{ij}(x,t)\pa_i\pa_j$ with the new weight function {$\vp_1=\e^{\la\psi_1}$}. 
Namely
\begin{align*}
&\int_D\left\{\f{1}{s\vp_1}|\pa_t u|^2+s\la^2\vp_1|\na u|^2
+ s^3\la^4\vp_1^3 u^2 \right\} \e^{2s\vp_1} \dx\dt
\nonumber\\
\le& C\int_D |L_0u|^2 \e^{2s\vp_1} \dx\dt
+\e^{C(\la)s}\int_{\pa D} (|\na u|^2 + |u|^2) \d S \dt
+ \e^{C(\la)s}\int_{\pa D\setminus\Sigma_0} |\pa_t u|^2 \d S\dt
\end{align*}
for all {$s\ge s_0$} and all $u\in H^{2,1}(D)$.

We note
$$
\sg(x,t) = \sum_{i,j=1}^n a_{ij}(x,t) (\pa_id) \pa_jd,\quad (x,t)\in\ov Q
$$
and
$$
w(x,t)=\e^{s\vp_1(x,t)} u(x,t)
$$
and
\begin{align*}
Pw= \e^{s\vp_1} L_0(\e^{-s\vp_1} w)
=&\pa_t w - \sum_{i,j=1}^n a_{ij}\pa_j\pa_j w + 2s\la\vp_1\sum_{i,j=1}^n a_{ij}(\pa_i d)\pa_j w
\\
&-s^2\la^2\vp_1^2\sg w+s\la^2\vp_1\sg w+s\la\vp_1 w\sum_{i,j=1}^n a_{ij}\pa_i\pa_j d-s\la\vp_1 w(\pa_t \psi_1).
\end{align*}

Now we introduce a new operator $P_3$ which is defined by
$$
P_3w:= Pw + \left(s\la^2\vp_1\sg - s\la\vp_1\sum_{i,j=1}^n a_{ij}\pa_i\pa_j d + s\la\vp_1 (\pa_t \psi_1)\right) w,
$$
and we decompose $P_3$ into the parts $P_1$ and $P_2$, 
$$
{  P_3w }= P_1w + P_2w,
$$
where 
$$
P_1w=- \sum_{i,j=1}^n a_{ij}\pa_j\pa_j w -s^2\la^2\vp_1^2\sg w,
$$
and
$$
P_2w=\pa_t w + 2s\la\vp_1\sum_{i,j=1}^n a_{ij}(\pa_i d)\pa_j w + 2s\la^2\vp_1\sg w.
$$
From the above notations for $P,P_1,P_2,P_3$, it follows that 
\begin{align*}
&\left\|\e^{s\vp_1}L_0u + \left(s\la^2\vp_1\sg - s\la\vp_1\sum_{i,j=1}^n a_{ij}\pa_i\pa_j d + s\la\vp_1 (\pa_t \psi_1)\right) w\right\|_{L^2(D)}^2 
\\
=& \|P_1w+P_2w\|_{L^2(D)}^2 
\ge 2\int_D (P_1w)(P_2w) \dx\dt + \|P_2w\|_{L^2(D)}^2.
\end{align*}
We will estimate $\int_D |P_2w|^2 + 2(P_1w)(P_2w) \dx\dt$ from below. 
Firstly, we have
\begin{align}
\label{esti-P1P2}
\int_D (P_1w)(P_2w) \dx\dt 
=& -\sum_{i,j=1}^n \int_D a_{ij}(\pa_i\pa_jw)\pa_tw \dx\dt
-\sum_{i,j=1}^n \int_D a_{ij}(\pa_i\pa_jw) 2s\la\vp_1\sum_{k,l=1}^n a_{kl} (\pa_kd)\pa_lw\dx\dt
\nonumber\\
&-\sum_{i,j=1}^n \int_D a_{ij}(\pa_i\pa_jw) 2s\la^2\vp_1\sg w \dx\dt
-\int_D s^2\la^2\vp_1^2 \sg w \pa_t w \dx\dt
\nonumber\\
&-\int_D 2s^3\la^3\vp_1^3 \sg w \sum_{i,j=1}^n a_{ij} (\pa_i d) \pa_j w \dx\dt
-\int_D 2s^3\la^4\vp_1^3\sg^2 w^2 \dx\dt
=:\sum_{k=1}^6 J_k.
\end{align}
Now, applying the integration by parts and the symmetry of $\{a_{ij}\}$: $a_{ij}=a_{ji}$, 
we give the estimates of $J_k$, $k=1,\cdots,6$ separately.
\begin{align*}
&J_1
= -\sum_{i,j=1}^n \int_D a_{ij}(\pa_i\pa_j w) \pa_tw \dx\dt 
\\
=& \sum_{i,j=1}^n \int_D (\pa_ia_{ij})(\pa_j w)\pa_tw \dx\dt 
+\sum_{i,j=1}^n \int_D a_{ij}(\pa_j w)\pa_i\pa_tw \dx\dt 
-\sum_{i,j=1}^n \int_{\pa D} a_{ij}(\pa_j w)\nu_i \pa_tw\d \Sg.
\end{align*}
Here and henceforth $\nu:=(\nu_1,\cdots,\nu_n,\nu_{n+1})$ denotes the unit normal exterior
with respect to the boundary $\pa D$ of $D$. In particular, $\nu_{n+1}$ is the component in the time direction. {  By noting $\nu_i=0, \forall i=1,...,n$ on $\Sg_0$, then integration by parts yields}
\begin{align*}
J_1
=& \sum_{i,j=1}^n \int_D (\pa_ia_{ij})(\pa_j w)\pa_tw \dx\dt 
+\f 1 2 \sum_{i,j=1}^n \int_D a_{ij}\pa_t\big((\pa_i w)\pa_jw\big) \dx\dt
\\
&-\sum_{i,j=1}^n \int_{\pa D\setminus \Sigma_0} a_{ij}(\pa_j w)\nu_i \pa_tw {  \d \Sg}
\\
=& \sum_{i,j=1}^n \int_D (\pa_ia_{ij})(\pa_j w)\pa_tw \dx\dt 
-\f 1 2 \sum_{i,j=1}^n \int_D (\pa_ta_{ij})(\pa_i w)\pa_jw \dx\dt 
\\
&+\f 1 2 \sum_{i,j=1}^n \int_{\pa D} a_{ij}(\pa_i w)(\pa_jw) \nu_{n+1} {  \d \Sg}
-\sum_{i,j=1}^n \int_{\pa D\setminus \Sigma_0} a_{ij}(\pa_j w) \nu_i\pa_tw {  \d \Sg}.
\end{align*}
Thus
\begin{align*}
|J_1|\le&C\int_D |\na w||\pa_t w| \dx\dt 
+ C\int_D |\na w|^2 \dx\dt
+C \int_{\pa D} |\na w|^2 {  \d \Sg}
+ C\int_{\pa D\setminus \Sigma_0} |\na w| |\pa_tw| {  \d \Sg}
\\
\le& C\int_D |\na w||\pa_t w| \dx\dt 
+ C\int_D |\na w|^2 \dx\dt
+C \int_{\pa D} |\na w|^2 {  \d \Sg}
+ C\int_{\pa D\setminus \Sigma_0} |\pa_tw|^2 {  \d \Sg}.
\end{align*}
Since the Cauchy-Schwarz inequality implies that
$$
|\na w||\pa_tw| 
= s^{\f 1 2} \la^{\f 1 2} \vp_1^{\f 1 2} |\na w|s^{-\f 1 2} \la^{-\f 1 2} \vp_1^{-\f 1 2}  |\pa_t w|
\le \f 1 2 s\la\vp_1 |\na w|^2 + \f 1 2 \f{1}{s\la\vp_1} |\pa_tw|^2,
$$
we have 
\begin{align*}
|J_1|\le& C\int_D \f{1}{s\la\vp_1}|\pa_t w|^2 \dx\dt + C \int_D s\la\vp_1 |\na w|^2 \dx\dt \\
&+ C\int_{\pa D} |\na w|^2 {  \d \Sg} +C\int_{\pa D\setminus\Sigma_0} |\pa_t w|^2 {  \d \Sg}.
\end{align*}

Next similar to the argument on pp. 12-13 in \cite{Y09}, we have
\begin{align*}
J_2 
=& -\sum_{i,j=1}^n \sum_{k,l=1}^n
 \int_D 2s\la\vp_1 a_{ij}a_{kl} (\pa_kd)(\pa_lw) \pa_i\pa_jw \dx\dt
\\
=&2s\la\int_D \sum_{i,j=1}^n \sum_{k,l=1}^n
 \la (\pa_id)\vp_1 a_{ij}a_{kl} (\pa_kd)(\pa_lw) \pa_jw \dx\dt
\\
&+2s\la\int_D \sum_{i,j=1}^n \sum_{k,l=1}^n
\vp_1 \pa_i(a_{ij}a_{kl} \pa_kd)(\pa_lw) \pa_jw \dx\dt
\\
&+2s\la\int_D \sum_{i,j=1}^n \sum_{k,l=1}^n
 \vp_1 a_{ij}a_{kl} (\pa_kd)(\pa_i\pa_lw) \pa_jw \dx\dt
\\
&-2s\la\int_{\pa D} \sum_{i,j=1}^n \sum_{k,l=1}^n \vp_1 a_{ij} a_{kl} (\pa_kd) ( \pa_l w) (\pa_jw) \nu_i {  \d \Sg}.
\end{align*}
We have 
$$
\mbox{(first term)} = 2s\la^2\int_D \vp_1 \left| \sum_{i,j=1}^n a_{ij}(\pa_i d)\pa_j w \right|^2  \dx\dt \ge 0
$$
and
\begin{align*}
\mbox{(third term)}
=&s\la\int_D \vp_1\sum_{i,j=1}^n\sum_{k,l=1}^n a_{ij}a_{kl} (\pa_kd)\pa_l\big((\pa_iw) (\pa_jw)\big) \dx\dt
\\
=&s\la\int_{\pa D} \vp_1 \sum_{i,j=1}^n \sum_{k,l=1}^n a_{ij}a_{kl} (\pa_kd) (\pa_iw) (\pa_jw) \nu_l {  \d \Sg}
\\
&-s\la^2\int_D \vp_1\sum_{i,j=1}^n \sg a_{ij} (\pa_iw) \pa_jw \dx\dt
\\
&-s\la\int_D \vp_1\sum_{i,j=1}^n \sum_{k,l=1}^n  \pa_l\big(a_{ij}a_{kl} (\pa_kd) \big)(\pa_iw) \pa_jw \dx\dt,
\end{align*}
which imply
\begin{align*}
J_2\ge -\int_D s\la^2\vp_1\sg \sum_{i,j=1}^n a_{ij} (\pa_i w)\pa_j w \dx\dt
- C\int_D s\la\vp_1 | \na w|^2 \dx\dt
-C\int_{\pa D} s\la\vp_1|\na w|^2 {  \d \Sg}.
\end{align*}
\begin{align*}
J_3 
=& 2\int_D s\la^2\vp_1\sg \sum_{i,j=1}^n a_{ij} (\pa_iw)\pa_jw \dx\dt
+ 2\int_D s\la^2 \sum_{i,j=1}^n \pa_i(\vp_1\sg a_{ij})  w\pa_jw \dx\dt
\nonumber\\
&- 2\int_{\pa D} s\la^2\vp_1\sg \sum_{i,j=1}^n a_{ij} w(\pa_jw)\nu_i {  \d \Sg}
\nonumber\\
 \ge& 2\int_D s\la^2\vp_1\sg \sum_{i,j=1}^n a_{ij} (\pa_iw)\pa_jw \dx\dt
- C\int_D s\la^3\vp_1 | \na w| | w| \dx\dt
- C\int_{\pa D} s\la^2\vp_1 | \na w| | w| {  \d \Sg}.
\end{align*}
By 
$$
s\la^3\vp_1 |\na w| |w| = (s\la^2\vp_1|w|)(\la |\na w|) 
\le \f 1 2 s^2\la^4\vp_1^2 w^2 + \f 1 2 \la^2|\na w|^2,
$$
and
$$
s\la^2\vp_1 |\na w| |w| = (s\la^{\f 3 2}\vp_1|w|)(\la^{\f 1 2} |\na w|) 
\le \f 1 2 s^2\la^3\vp_1^2 w^2 + \f 1 2 \la|\na w|^2,
$$
we have
\begin{align*}
J_3 
 \ge& 2\int_D s\la^2\vp_1\sg \sum_{i,j=1}^n a_{ij} (\pa_iw)\pa_jw \dx\dt
- C\int_D s^2\la^4\vp_1^2 w^2 \dx\dt
- C\int_D \la^2 |\na w|^2 \dx\dt
\\
&- C\int_{\pa D} \la | \na w|^2 {  \d \Sg}
- C\int_{\pa D} s^2\la^3\vp_1^2 w^2 {  \d \Sg}.
\end{align*}

\begin{align*}
&|J_4| 
=\left| -\f 1 2\int_D s^2\la^2\vp_1^2 \sg \pa_t(w^2) \dx\dt\right| \\
=& \left|\int_D s^2\la^3\vp_1^2\be(2\al_1-2)t^{1-2\al_1}\sg w^2 \dx\dt
+\f 1 2 \int_D s^2\la^2\vp_1^2(\pa_t\sg) w^2 \dx\dt
-\f 1 2 \int_{\pa D} s^2\la^2\vp_1^2\sg w^2\nu_{n+1} {  \d \Sg} \right|
\\
\le& C\int_D s^2\la^3\vp_1^2 w^2 \dx\dt +C\int_{\pa D}s^2\la^2\vp_1^2w^2 {  \d \Sg} .
\end{align*}
\begin{align*}
J_5
=& -\int_D s^3 \la^3 \vp_1^3 \sum_{i,j=1}^n \sg a_{ij}(\pa_i d)\pa_j(w^2) \dx\dt \\
=&3\int_D s^3\la^4\vp_1^3\sg^2w^2 \dx\dt
+\int_D s^3\la^3\vp_1^3\sum_{i,j=1}^n \pa_j(\sg a_{ij} \pa_id) w^2\dx\dt \\
&-\int_{\pa D} s^3\la^3\vp_1^3 \sum_{i,j=1}^n \sg a_{ij} (\pa_i d)w^2\nu_j {  \d \Sg} \\
\ge& 3\int_D s^3\la^4\vp_1^3\sg^2 w^2 \dx\dt  - C\int_D s^3\la^3\vp_1^3w^2 \dx\dt
-C\int_{\pa D} s^3\la^3\vp_1^3 w^2 {  \d \Sg}.
\end{align*}
\begin{align*}
J_6 = -\int_D 2s^3 \la^4 \vp_1^3 \sg^2 w^2 \dx\dt.
\end{align*}
By the definition of $P_2$, we have
\begin{align*}
\ep\int_D \f1{s\vp_1}|\pa_tw|^2 \dx\dt
=&\ep\int_D \f1{s\vp_1}
\left| P_2w - 2s\la\vp_1\sum_{i,j=1}^n a_{ij}(\pa_id)\pa_jw - 2s\la^2\vp_1\sg w\right|^2 \dx\dt \\
\le& C\int_D |P_2w|^2 \dx\dt  + C\ep\int_D s\la^2\vp_1 |\na w|^2 \dx\dt + C\ep\int_D s\la^4\vp_1 w^2 \dx\dt.
\end{align*}
By summing up all the above estimates for $J_k$, $k=1,\cdots,6$, we find
\begin{align*}
&\int_D s^3\la^4\vp_1^3\sg^2 w^2 \dx\dt
+ \int_D s\la^2\vp_1\sg \sum_{i,j=1}^n a_{ij} (\pa_i w)\pa_j w \dx\dt
+ \left(\ep-\f{C}{\la}\right)\int_D \f{1}{s\vp_1}|\pa_tw|^2 \dx\dt \\
\le& C\int_D |L_0u|^2e^{2s\vp_1} \dx\dt 
+ C\int_D (s\la\vp_1 + \ep s\la^2\vp_1 + \la^2)|\na w|^2 \dx\dt
+ C\int_D (s^3\la^3\vp_1^3+s^2\la^4\vp_1^2)w^2 \dx\dt \\
&+C\int_{\pa D} (s\la\vp_1 |\na w|^2 + s^3\la^3\vp_1^3 w^2) {  \d \Sg} 
+C\int_{\pa D\setminus\Sg_0} |\pa_tw|^2 {  \d \Sg},
\end{align*}
which combined with the ellipticity of $a_{ij}$ and $\sg_0:=\inf_{(x,t)\in Q} \sg(x,t)>0$ yields that
\begin{align*}
&\int_D s^3\la^4\vp_1^3\sg_0^2w^2 dxdt
+\int_D {(\sg_0\rho-C\ep)} s\la^2\vp_1 |\na w|^2 \dx\dt 
+ \left(\ep-\f{C}{\la}\right)\int_D \f{1}{s\vp_1}|\pa_tw|^2 \dx\dt
\\
\le& C\int_D |L_0u|^2e^{2s\vp_1} \dx\dt
+ C\int_D (s^3\la^3\vp_1^3 + s^2\la^4\vp_1^2) w^2 \dx\dt 
+ C\int_D (s\la\vp_1+\la^2) |\na w|^2 \dx\dt
\nonumber\\
&+\int_{\pa D} (s\la\vp_1|\na w|^2 + s^2\la^3\vp_1^2w^2) {  \d \Sg}
+\int_{\pa D\setminus\Sg_0} |\pa_t w|^2 {  \d \Sg}.
\end{align*}
Thus choosing $\ep>0$ small, and choosing $\la$ and then $s$
large, we can absorb terms suitably to obtain
\begin{align*}
&\int_D \left\{\f{1}{s\vp_1}| \pa_tw|^2 + s\la^2\vp_1| \na w|^2+ s^3\la^4\vp_1^3w^2 \right\} \dx\dt
\\
\le& C\int_D |L_0u|^2 \e^{2s\vp_1} \dx\dt
+\int_{\pa D} (s\la\vp_1|\na w|^2 + s^2\la^3\vp_1^2w^2) {  \d \Sg}
+\int_{\pa D\setminus\Sg_0} |\pa_t w|^2 {  \d \Sg}.
\end{align*}
Noting $w=u\e^{s\vp_1}$, we have
\begin{align*}
&\int_D \left\{\f{1}{s\vp_1}| \pa_tu|^2 + s\la^2\vp_1| \na u|^2+ s^3\la^4\vp_1^3u^2 \right\} \e^{2s\vp_1}\dx\dt
\\
\le& C\int_D |L_0u|^2 \e^{2s\vp_1} \dx\dt
+{C(\la)}\e^{C(\la)s}\int_{\pa D} (|\na u|^2 + u^2) {  \d \Sg}
+{C(\la)}\e^{C(\la)s}\int_{\pa D\setminus\Sg_0} |\pa_t u|^2 {  \d \Sg},
\end{align*}
which completes the proof of the theorem.
\end{proof}

\subsection{Application to a lateral Cauchy problem for the sub-diffusion}
\label{sec-LCP<0.5}
In this section, we will give a proof of Theorem \ref{thm-LCP<0.5}. To prove this, we follow the usual argument 
used in analyzing the lateral Cauchy problem for the parabolic equations, that is, we use the Carleman estimate 
derived in Section \ref{sec-CE<0.5}. The first problem which we have to overcome is to evaluate the fractional derivative by the first order time-derivative under some suitable norm.  
Namely, the following lemma holds.
\begin{lem}
\label{lem-esti-frac}
Let $T>0$ and $0<\al\le\al_1<\f 1 2$ be given constants, 
then the following inequality
\begin{equation}
\label{esti-frac}
\int_{\cC_1} |\pa_t^\al u|^2 \e^{2s\vp_1} \dx\dt
\le C\int_{\cC_2}  \f1{s\la\vp_1}|\pa_t u|^2 \e^{2s\vp_1}\dx\dt
\end{equation}
holds true for all $u\in H^{2,1}(Q)$, where $\vp_1=\e^{\la\psi_1}$ with $\psi_1(x,t)=d(x)-\be t^{2-2\al_1}$, and 
$\cC_i:=\{(x,t);\ x\in\ov\Om,\ t>0,\ \vp_1(x,t)>c_i\}$, $i=1,2$, and $c_i$ are positive constants such that 
$c_2<c_1$.
\end{lem}
\begin{proof}
We choose a nonnegative function {$\Phi\in L^\infty(\R^{n+1})$} such that supp$\Phi\subset \cC_2$ and $\Phi\equiv1$ in $\cC_1$. Thus we have
$$
\int_{\cC_1} |\pa_t^\al u|^2 \e^{2s\vp_1} \dx \dt
=\int_{\cC_1} \Phi(x,t) |\pa_t^\al u|^2 \e^{2s\vp_1} \dx \dt,
$$
which combined with the definition of the Caputo derivative implies that
$$
\int_{\cC_1} {  \Phi(x,t)}|\pa_t^\al u|^2 \e^{2s\vp_1} \dx \dt
 {  =} \int_{\cC_1}\left|\f{\Phi^{\f 1 2}(x,t)}{\Ga(1-\al)} \int_0^t(t-r)^{-\al} \pa_r u(x,r)\d r\right|^2
 \e^{2s\vp_1}\dx\dt.
$$
Moreover, since for any fixed $x\in\Om$, $\vp_1(x,t)$ is decreasing with respect to the variable $t>0$, we see that $\vp_1(x,t)\le\vp_1(x,r)$ for any $x\in\Om$ and $0<r\le t$, so that $(x,t)\in\cC_1$ implies that $(x,r)\in\cC_1$ for $0<r\le t$, hence that $\Phi(x,r)=1$ if $(x,t)\in\cC_1$ and $0<r<t$,
finally we have
\begin{align*}
\int_{\cC_1} \Phi(x,t)|\pa_t^\al u|^2 \e^{2s\vp_1} \dx \dt
{  =} &\int_{\cC_1}\left|\f1{\Ga(1-\al)} \int_0^t(t-r)^{-\al} \Phi^{\f12}(x,r)\pa_r u(x,r)\d r\right|^2 \e^{2s\vp_1}\dx\dt
 \\
\le& \int_Q \left|\f1{\Ga(1-\al)} \int_0^t(t-r)^{-\al} \Phi^{\f12}(x,r)\pa_r u(x,r)\d r\right|^2 \e^{2s\vp_1}\dx\dt.
\end{align*}
Now by noting that $\vp_1(x,t) \ge c_0$, where $c_0>0$ is a constant, we have 
$$
\pa_t\psi_1 = \be(2\al_1-2) t^{1-2\al_1}, \quad
\pa_t\vp_1 = \la(\pa_t\psi_1)\vp_1 = (2\al_1-2)\be\la\vp_1 t^{1-2\al_1}.
$$
Hence
\begin{equation}
\label{equ-phi}
t^{1-2\al_1}e^{2s\vp_1}
=-\f{1}{4\be s\la\vp_1(1-\al_1)}\pa_t(e^{2s\vp_1}).
\end{equation}
By the Cauchy-Schwarz inequality and (3), we have
\begin{align*}
&\int^T_0 \left| \int^t_0 (t-r)^{-\al}\Phi^{\f12}(x,r) \pa_ru(r) \d r\right|^2 
\e^{2s\vp_1(x,t)} \dt
\\
\le& \int^T_0 \left(\int^t_0 (t-r)^{-2\al} \d r\right)
\left(\int^t_0 \Phi(x,r)|\pa_ru(r)|^2 \d r \right) \e^{2s\vp_1(x,t)} \dt\\
=& \f{1}{1-2\al}\int^T_0 t^{1-2\al}
\left(\int^t_0 \Phi(x,r)|\pa_r u(r)|^2 \d r\right) \e^{2s\vp_1(x,t)} \dt
\end{align*}
Moreover, since $0<\al<\al_1$, we see that
\begin{align*}
&\int_0^T \left| \int^t_0 (t-r)^{-\al} \Phi^{{  \f 1 2}}(x,r)\pa_r u(r) \d r\right|^2 
\e^{2s\vp_1} \dt
\le \f{T^{2(\al_1-\al)}}{1-2\al}\int_0^T t^{1-2\al_1}
\left(\int_0^t \Phi(x,r)|\pa_r u(r)|^2 \d r\right) \e^{2s\vp_1} \dt.
\end{align*}
Now from the formula \eqref{equ-phi}, integration by parts implies 
\begin{align*}
&\int_0^T t^{1-2\al_1}
\left(\int_0^t \Phi(x,r)|\pa_r u(r)|^2 \d r\right) \e^{2s\vp_1(x,t)} \dt
\\
= &\f{1}{1-\al_1} \left(\f{-1}{4\be s\la\vp_1} \int_0^t \Phi(x,r)|\pa_r u(r)|^2 \d r \right) \e^{2s\vp_1(x,t)} \Big|_{t=0}^{t=T} 
+\f{1}{1-\al_1}\int_0^T \f{\Phi(x,t)}{4\be s\la\vp_1}|\pa_t u|^2 
\e^{2s\vp_1} \dt
\\
+& \int_0^T \f{t^{1-2\al_1}}{2s\vp_1}
\left(\int_0^t \Phi(x,r)|\pa_ru(r) |^2 \d r \right) \e^{2s\vp_1}\dt
\\
\le&  \f{1}{1-\al_1}\int_0^T \f{\Phi(x,t)}{4\be s\la\vp_1}|\pa_t u|^2 \e^{2s\vp_1} \dt
+\int_0^T \f{t^{1-2\al_1}}{2s\vp_1}
\left(\int^t_0 \Phi(x,r)|\pa_r u(r)|^2 dr \right) \e^{2s\vp_1}\dt.
\end{align*}
The last term on the right-hand side can be absorbed 
into the left-hand side by choosing $s>0$ large and 
we have
$$
\int_0^T t^{1-2\al_1} 
\left(\int_0^t \Phi(x,r) |\pa_r u(r)|^2 \d r\right) \e^{2s\vp_1(x,t)} \dt 
\le C\int_0^T \f{\Phi(x,t)}{s\la\vp_1} |\pa_t u(t)|^2 \e^{2s\vp_1} \dt.
$$
Thus
$$
\int_Q t^{1-2\al_1} 
\left(\int_0^t \Phi(x,r) |\pa_r u(r)|^2 \d r\right) \e^{2s\vp_1(x,t)} \dx\dt 
\le C\int_Q \f{\Phi(x,t)}{s\la\vp_1} |\pa_t u(t)|^2 \e^{2s\vp_1} \dx\dt,
$$
which combined with the fact supp$\Phi\subset \cC_2$ implies that
$$
\int_{\cC_1} |\pa_t^{\al} u|^2 \e^{2s\vp_1} \dx\dt
\le C\int_{\cC_2} \f1{s\la\vp_1} | \pa_tu|^2 \e^{2s\vp_1} \dx\dt,
$$
which completes the proof of the lemma.
\end{proof}


Before giving the proof of Theorem \ref{thm-LCP<0.5}, we introduce some notations.

For arbitrary given domain $\Om_0$ such that $\ov{\Om_0}\subset\Om$, 
similar to Theorem 5.1 in \cite{Y09}, 
we will choose a suitable weight function $\psi_1(x,t):=d(x)-\be t^{2-2\al_1}$. 
For this, we first choose a bounded domain $\Om_1$ with smooth boundary such that
$$
\Om\subsetneq\Om_1,
\quad \ov\Ga=\ov{\pa\Om\cap\Om_1}, 
\quad \pa\Om\setminus\Ga\subset \pa\Om_1.
$$
We then apply Lemma 4.1 in \cite{Y09} to obtain 
$d\in C^2(\ov{\Om_1})$ satisfying
$$
d(x)>0,\ x\in\Om_1,
\quad d(x)=0,\ x\in\pa\Om_1,
\quad |\na d(x)|>0,\ x\in\ov\Om.
$$
Then we can choose $\be>0$ such that

\begin{equation}
\label{condi-beta}
\be(\tfrac{T}2)^{2-2\al_1} < \|d\|_{C(\ov{\Om_1})}<\be T^{2-2\al_1}.
\end{equation}
Moreover, since $\ov{\Om_0}\subset\Om_1$, we can choose 
a sufficiently large $N>1$ such that
\begin{equation}
\label{condi-Omega_0}
\Om_0\subset 
\ov\Om \cap \{x\in\Om_1;\ 
d(x)> \f4N \|d\|_{C(\ov{\Om_1})}\}.
\end{equation}

We set $\mu_k=\exp\{\la(\f k N \|d\|_{C(\ov{\Om_1})}-\f{\be(\frac{T}2)^{2-2\al_1}}N)\}$, and 
$D_k:=\{(x,t);\ x\in\ov\Om,\ t>0,\ \vp(x,t)>\mu_k\}$, $k=1,2,3,4$. Then we can verify from \eqref{condi-beta} and 
\eqref{condi-Omega_0} that
\begin{equation}
\label{condi-D}
\Om_0\times(0,\f{T}{2M}) {  \subset D_4}\subset D_3\subset D_1
\subset \ov\Om \times (0,T),
\end{equation}
where $M:=N^{\f1{2-2\al_1}}$, and 
\begin{equation}
\label{condi-boundary_of_D}
\pa D_1 \subset \Sigma_0\cup\Sigma_1\cup\Sigma_2
\end{equation}
{are} valid. Here $\Sigma_0=\{(x,0);\ x\in\ov\Om\}$, $\Sigma_1\subset \Ga\times(0,T)$
and $\Sigma_2=\{(x,t);\ x\in\Om,\ t>0, \vp(x,t)=\mu_1\}$.

\bigskip

Now we are ready to give the proof of our main theorem.

\begin{proof}[\bf Proof of Theorem \ref{thm-LCP<0.5}]
We start from the Cauchy problem
$$
\left\{
\begin{alignedat}{2}
&u(x,t) = g_0(x,t) &\quad& \mbox{on $\Ga\times(0,T]$},\\
&\pa_{\nu_A} u(x,t) = g_1(x,t) &\quad& \mbox{on $\Ga\times(0,T]$}
\end{alignedat}
\right.
$$
for the equation \eqref{equ-gov}. 

Henceforth $C>0$ denotes generic constants depending on $\la$, 
but independent of $s$ and the choice of $g_0$, $g_1$, $u$. 
For it, we need a cut-off function because we have no data 
$\pa_{\nu_A} u$ on $\pa D\setminus\Ga\times(0,T)$. Let 
$\chi\in C^\infty(\R^{n+1})$ such that $0\le\chi\le1$ and
\begin{equation}
\label{condi-chi}
\chi(x,t)=
\begin{cases}
1, &\vp_1(x,t)>\mu_3,\\
0, &\vp_1(x,t)<\mu_2.
\end{cases}
\end{equation}

Setting $v:=\chi u$, $\wt L:=L-\sum_{j=1}^\ell q_j\pa_t^{\al_j}$, 
and then using Leibniz's formula for the differential of the 
product we have
\begin{equation}
\label{equ-v}
\wt Lv{  = \chi\wt Lu +A_1 u}= \chi L u - \chi\sum_{j=1}^\ell q_j\pa_t^{\al_j} u+A_1 u
=\chi F- \chi\sum_{j=1}^\ell q_j\pa_t^{\al_j} u+A_1 u.
\end{equation}
Here the last term $A_1 u$ involves only the linear combination 
of $(\pa_t\chi) u$, $(\pa_i\pa_j\chi) u$, $(\pa_i\chi)(\pa_j  u)$ 
and $(\pa_i\chi) u$, $i,j=1,\cdots,n$. 

By \eqref{condi-boundary_of_D} and \eqref{condi-chi}, 
we see that $v=|\na v|{  =\pa_t v}=0 $ on $\Sigma_2$. 
Hence using the Carleman estimate in Theorem \ref{thm-CE<0.5},
from $D_3\subset D_1$ by an argument similar to Theorem 3.2 
in \cite{Y09} in $D_1$ to \eqref{equ-v}, we find
\begin{align}
\label{esti-carleman}
&\int_{D_3}\left\{\f 1 {s{  \vp_1}} |\pa_t v|^2 
+s\la^2\vp_1|\na v|^2 +s^3\la^4\vp_1^3v^2\right\}\e^{2s\vp_1}\dx\dt 
\nonumber\\
\le&\int_Q F^2 \e^{2s\vp_1}\dx\dt
+C\int_{D_1}\sum_{j=1}^\ell \chi^{  2} (x,t)|\pa_t^{\al_j}u|^2\e^{2s\vp_1}\dx\dt
+C\int_{D_1} |A_1u|^2 \e^{2s\vp_1} \dx\dt
\nonumber\\
&+\e^{C(\la)s} \int_{\Sg_0\cup(\Ga\times(0,T))} (|\na v|^2+v^2){  \d \Sg}
+ \e^{C(\la)s} \int_{\Ga\times(0,T)} |\pa_t v|^2 \d S \dt.
\end{align}
for all $s\ge s_0$ and $\la\ge\la_0$. 

By \eqref{condi-chi}, $A_1u$ does not vanish only if 
$\mu_2\le\vp(x,t)\le \mu_3$ and so
$$
\int_{D_1} |A_1u|^2 \e^{2s\vp_1} \dx\dt
\le C\e^{2s\mu_3}\|u\|_{H^{1,0}(Q)}^2.
$$
Moreover, from \eqref{condi-D} and Lemma \ref{lem-esti-frac}, letting $\cC_1=\left\{(x,t); x\in\ov\Om, t>0, \vp_1(x,t) > \f{\mu_3+\mu_4}2\right\}$ and $\cC_2=D_3$ in Lemma \ref{lem-esti-frac}, for $\la$ being large {  fixed}, we conclude that the integration
$
\int_{\cC_1} \sum_{j=1}^\ell \chi^{  2}(x,t)|\pa_t^{\al_j} u|^2 \e^{2s\vp_1}\dx\dt
$
can be absorbed by the left-hand side of \eqref{esti-carleman} , which implies
\begin{align*}
&\int_{D_3}\left\{\f1{s\vp_1} |\pa_t u|^2
+s \vp_1 |\na u|^2 +s^3 \vp_1^3 u^2\right\} \e^{2s\vp_1}\dx\dt 
\\
\le& C\e^{Cs}\|F\|_{L^2(Q)}
+C\e^{2s\mu_3}\|u\|_{H^{1,0}(Q)}^2
+C\sum_{j=1}^\ell \int_{D_1\setminus \cC_1} |\pa_t^{\al_j}u|^2 \e^{2s\vp_1} \dx\dt 
\nonumber\\
&+\e^{Cs} \int_{\Sg_0\cup(\Ga\times(0,T))} (|\na v|^2+v^2) {  \d \Sg}
+ \e^{Cs} \int_{\Ga\times(0,T)} |\pa_t v|^2\d S\dt.
\end{align*}

By \eqref{condi-Omega_0}, we can directly verify that 
$\vp_1(x,t)\le\f{\mu_3+\mu_4}2$ in $D_1\setminus \cC_1$, and if 
$(x,t)\in\Om_0\times(0,\ve)$, then $\vp_1(x,t)>\mu_4$. 
Then combined with \eqref{condi-D} and \eqref{condi-boundary_of_D},  
{  by H\"{o}lder's inequality}, we have
\begin{align*}
&\e^{2s\mu_4}\int_0^{\f{T}{2M}} \int_{\Om_0} 
\left\{\f 1 s |\pa_t u|^2
+s|\na u|^2 +s^3 u^2\right\}\dx\dt 
\\
\le&C\e^{Cs}\|F\|_{L^2(Q)}
+C\e^{2s\mu_3}\|u\|_{H^{1,0}(Q)}^2
+C\e^{2s\f{\mu_3+\mu_4}2}\int_Q |\pa_t u|^2 \dx\dt
 \nonumber\\
&+\e^{Cs} \int_{\Om} (|\na v(x,0)|^2+v^2(x,0))\dx
+ \e^{Cs} \int_{\Ga\times(0,T)}(|\pa_t v|^2+|\na v|^2+v^2)\d S\dt.
\end{align*}
for $s\ge s_0$.
Then dividing both sides by $\e^{2s\mu_4}$, since 
$$
s\e^{-2s\f{\mu_4-\mu_3}2} \le C\e^{-\f{s(\mu_4-\mu_3)}2},\quad 
s\e^{-2s(\mu_4-\mu_3)} \le C\e^{-s(\mu_4-\mu_3)} \le C\e^{-\f{s(\mu_4-\mu_3)}2},
$$
by replacing $C$ by $C\e^{Cs_0}$, we have
\begin{equation}
\label{esti-u}
\|u \|_{H^{1,1}(\Om_0\times(0,\f{T}{2M}))}^2
\le C\e^{-\f{\mu_4-\mu_3}2s}\|u\|_{H^{1,1}(Q)}^2
+C\e^{Cs} \cD^2
\end{equation}
for all {$s\ge 0$} and $u\in H^{2,1}(Q)$, where the constant $C>0$ depends on $T,\Om_0$ and the coefficients of the equation \eqref{equ-gov}. 

Firstly, if $\cD=0$, letting $s\to\infty$, we conclude that $u=0$ in $\Om_0\times(0,\f{T}{2M})$,
so that the conclusion of Theorem \ref{thm-LCP<0.5} holds true.
Next let $\cD\ne0$. First let $\cD\ge\|u\|_{H^{1,1}(Q)}$.
Then \eqref{esti-u} implies 
$$
\|u \|_{H^{1,1}(\Om_0\times(0,\f{T}{2M}))}
\le C\e^{Cs}\cD,\quad s\ge 0,
$$
which already proves the theorem.
Second let $\cD\le \|u\|_{H^{1,1}(Q)}$, 
we choose $s>0$ minimizing the right-hand side of \eqref{esti-u},
that is
$$
\e^{-\f{\mu_4-\mu_3}2s}\|u \|_{H^{1,1}(Q)}^2 
=\e^{Cs}\cD^2.
$$
By $\cD\ne0$, we can choose
$$
s=\f4{2C+\mu_4-\mu_3}\log \f{\|u \|_{H^{1,1}(Q)}}{\cD}>0.
$$
Then \eqref{esti-u} gives
$$
\|u \|_{H^{1,1}(\Om_0\times(0,\f{T}{2M}))}
\le 2C\|u \|_{H^{1,1}(Q)}^{1-\te} \cD^\te,
$$
where $\te:=\f{\mu_4-\mu_3}{2C+\mu_4-\mu_3}$, and the constant $C$ depends on $\Om_0,T$ and the coefficients of the equation \eqref{equ-gov}. 
We complete the proof of the theorem by setting $\ve=\f{T}{2M}$.

\end{proof}

\section{Carleman estimate for a sup-diffusion and its applications}
\label{sec-order<=0.75}

In this section, we pay attention to the Carleman estimate for the equation \eqref{equ-gov2} in the case of 
$\al=\f{m}{k}$, say, 
\begin{equation}
\label{sy:<=0.75Caputo}
\pa_t u + q(x)\pa_t^{\f{m}{k}} u - \Delta u + B\cdot \na u + {c}u = F \quad \mbox{in }Q,
\end{equation}
and its applications to the lateral Cauchy problem and the inverse source problem. Without loss of generality, we set $a_{ij}=\de_{ij}$, $1\le i,j\le n$ here. In the following 
two subsections, we will give the proofs of Theorems \ref{thm-CE<=0.75}, \ref{thm-LCP=0.75} and 
\ref{thm-ISP=0.75} respectively.

\subsection{Carleman estimate for a sup-diffusion}
\label{sec-CE<=0.75}

In this subsection, we will give a proof of Theorem \ref{thm-CE<=0.75}. Similar to the case of $\al<\f12$, we construct the Carleman estimate for the following parabolic type equation
\begin{equation}
\label{CE:eq1}
\pa_t u - \Delta u + B\cdot\na u + {c} u = F - q(x)\pa_t^{\al} u,
\end{equation}
whereas here we further multiply on both sides of the above equation by several Riemann-Liouville fractional derivatives in order to dealing the new source term $\pa_t^{\al}u$. We have the following details of the proof.

\begin{proof}[\bf Proof of Theorem \ref{thm-CE<=0.75}]

For some technical reasons, we divide the proof into two cases:
\vspace{0.2cm}

(\romannumeral1) $\al = \f{m}{2k+1}$, $m,k\in \Z$, $k\ge 1$, $m=1,2,...,2k$,

(\romannumeral2) $\al = \f{m}{2k}$, $m,k\in \Z$, $k\ge 1$, $m=1,2,...,2k-1$.
\vspace{0.2cm}

In the first part, we consider the case (\romannumeral1) in which the denominator is odd. 
Because of the zero initial condition, the equation \eqref{CE:eq1} can be rephrased as
\begin{equation}
\label{sy:<=0.75RLD}
\wt L u := \pa_t u - \Delta u + B\cdot \na u + {c}u = F - q(x)D_t^{\f{m}{2k+1}} u \quad \mbox{in }Q.
\end{equation}
Recalling the definition of Riemann-Liouville fractional integral operator $D_t^{-p}$:
$$
D_t^{-p} u := _0\!D_t^{-p} u = \f1{\Ga(p)}\int_0^t (t-s)^{p-1} u(s) \d s,
$$
and we have the semigroup property for the R.- L. fractional integral operators:
$$
D_t^{-p_1} D_t^{-p_2}u = D_t^{-p_1-p_2} u,\quad \forall p_1,\ p_2>0.
$$
We first apply the fractional differential(or integral) operators $D_t^{\f{j}{2k+1}}$ to the equation \eqref{sy:<=0.75RLD} separately for $j=-k,...,k$ to derive
$$
D_t^{\f{j}{2k+1}} (\pa_t u) - D_t^{\f{j}{2k+1}}(\De u) + D_t^{\f{j}{2k+1}}(B\cdot\na u) + D_t^{\f{j}{2k+1}}({c}u) = D_t^{\f{j}{2k+1}} F - D_t^{\f{j}{2k+1}}(qD_t^{\f{m}{2k+1}} u),\quad j=-k,...,k.
$$
Moreover, the homogeneous initial value implies that the differential operators and the R.-L. integral operator are commutable, which along with the formula $D_t^{\f{j}{2k+1}}(D_t^{\f{m}{2k+1}} u) = D_t^{\f{j}{2k+1}+\f{m}{2k+1}} u = D_t^{\f{j+m}{2k+1}} u$ implies that
\begin{equation}
\label{CE:eq2}
\wt L (u_j) = \pa_t u_j - \Delta u_j + B\cdot\na u_j + {c} u_j = D_t^{\f{j}{2k+1}} F - qD_t^{\f{j+m}{2k+1}} u, \quad j=-k,...,k
\end{equation}
where $u_j := D_t^{\f{j}{2k+1}}u$, for all $j\in \Z$. Next we denote $\wt{u_j} := \chi_0 u_j$. 
Then by noting that
$$
\wt L(\chi_0 u) - \chi_0 \wt Lu = (\pa_t\chi_0) u -2\na\chi_0\cdot\na u - (\De\chi_0) u + (B\cdot\na\chi_0) u,
$$
the equations \eqref{CE:eq2} can be rewritten by:
\begin{equation}
\label{CE:eq3}
\wt L(\wt{u_j}) = \chi_0 D_t^{\f{j}{2k+1}}F - (\chi_0 q) u_{j+m} - 2\na\chi_0\cdot\na u_j  + (B\cdot\na\chi_0 - \Delta\chi_0 + \pa_t\chi_0) u_j
\end{equation}
for all $j=-k,...,k$.
Now we use a Carleman estimate for parabolic type stated in the following lemma.
\begin{lem}
\label{CE:para}
Let $F\in L^2(Q)$. Then there exist constants $\hat{\la} \ge 1$, $\hat{s} \ge 1$ and $C>0$, $C(\la)>0$ such that
\begin{align*}
&\int_Q \bigg\{s^{\tau-1}\la^{\tau}\vp_2^{\tau-1}\Big(|\pa_t u|^2 + \sum_{i,j=1}^n|\pa_i\pa_j u|^2\Big) + s^{\tau +1}\la^{\tau +2}\vp_2^{\tau +1}|\na u|^2 + s^{\tau +3}\la^{\tau +4}\vp_2^{\tau +3}|u|^2\bigg\}e^{2s\vp_2} \dx\dt \\
&\hspace{3cm} \le C\int_Q (s\la\vp_2)^{\tau}|F|^2e^{2s\vp_2}\dx\dt + C(\la)e^{C(\la)s}s^{\tau}\int_{\pa Q} (|\na_{x,t} u|^2 + |u|^2) \d S\dt
\end{align*}
for all $s \ge \hat{s}$, $\la \ge \hat{\la}$, $\tau\in \Z$ and all $u$ smooth enough satisfying the equation:
$$
\wt Lu = F.
$$
\end{lem}
\noindent The proof of this lemma is similar to that of Theorem 3.2 in \cite{Y09}. However we need to calculate more carefully when we do the integration by parts because we have a new weight $\vp_2^{\tau}$. To specify the dependence, $C$ and $\hat{\la}$ depends on $\|\na d\|$ while $\hat{s}$ also depends on $\tau,\be,T$ and the coefficients of the operator $\wt L$. Furthermore, $C(\la)$ even depends on large parameter $\la$ but is independent of $s$. 
\vspace{0.2cm}

Applying the above lemma to the equations \eqref{CE:eq3}, we have the following Carleman estimates:
\begin{align}
\nonumber
&\la\int_Q \bigg\{(s\la\vp_2)^{\tau_j-1}\Big(|\pa_t \wt{u_j}|^2 + \sum_{i,l=1}^n|\pa_i\pa_l \wt{u_j}|^2\Big) + (s\la\vp_2)^{\tau_j +1}|\na \wt{u_j}|^2 + (s\la\vp_2)^{\tau_j +3}|\wt{u_j}|^2\bigg\}e^{2s\vp_2} \dx\dt \\
\nonumber
&\hspace{0cm} \le C\int_Q (s\la\vp_2)^{\tau_j}(|\chi_0 D_t^{\f{j}{2k+1}}F|^2 + |\chi_0 u_{j+m}|^2)e^{2s\vp_2}\dx\dt + C\int_Q (s\la\vp_2)^{\tau_j}( |\wt L(\chi_0 u_j) - \chi_0 \wt L(u_j)|^2 )e^{2s\vp_2}\dx\dt \\
\label{CE:eq4}
&\hspace{1cm} + C(\la)s^{\tau_j}e^{C(\la)s}\int_{\pa Q} (|\na_{x,t}(\chi_0 u_j)|^2 + |\chi_0 u_j|^2)\d S\dt
\end{align}
for all $\la\ge \la_j$, $s\ge s_j$ and $j= -k,..,k$.

Moreover, we observe that a direct calculation implies that
\begin{align*}
&|\pa_t \wt{u_j}|^2 = |(\pa_t\chi_0) u_j + \chi_0(\pa_t u_j)|^2 \ge \f12|\chi_0 (\pa_t u_j)|^2 - |(\pa_t\chi_0)u_j|^2, \\
&|\pa_i\pa_l\wt{u_j}|^2 = |\chi_0\pa_i\pa_l u_j + (\pa_i\chi_0)(\pa_l u_j) + (\pa_l\chi_0)(\pa_i u_j) + (\pa_i\pa_l\chi_0)u_j|^2 \\
&\hspace{1.2cm} \ge \f12|\chi_0\pa_i\pa_l u_j|^2 - 3|(\pa_i\chi_0)(\pa_l u_j)|^2 - 3|(\pa_i\chi_0)(\pa_l u_j)|^2 - 3|(\pa_i\pa_l\chi_0)u_j|^2, \\
&|\na \wt{u_j}|^2 = |\chi_0 \na u_j + u_j \na\chi_0|^2\ge \f12|\chi_0\na u_j|^2 - |\na\chi_0|^2|u_j|^2.
\end{align*}
Thus, Carleman inequalities \eqref{CE:eq4} lead to
\begin{align}
\nonumber
&\la\int_Q \chi_0^2\bigg\{ (s\la\vp_2)^{\tau_0 -1}\Big(|\pa_t u|^2 + \sum_{i,l=1}^n|\pa_i\pa_l u|^2\Big) + (s\la\vp_2)^{\tau_0 +1}|\na u|^2 + (s\la\vp_2)^{\tau_0 + 3}|u|^2 \bigg\}e^{2s\vp_2}\dx\dt \\
\label{CE:eq5}
&\hspace{0cm} \le C\int_Q\chi_0^2(s\la\vp_2)^{\tau_0}(|F|^2 + |u_{m}|^2)e^{2s\vp_2}\dx\dt + C(\la)s^{\tau_0}e^{C(\la)s}\int_{\pa Q} (|\na_{x,t}(\chi_0 u)|^2 + |\chi_0 u|^2)\d S\dt \\
\nonumber
&\hspace{0cm} + C\int_Q (s\la\vp_2)^{\tau_0}\Big( |\wt L(\chi_0 u) - \chi_0 \wt Lu|^2 + (|\pa_t\chi_0|^2 + s\la^2\vp_2|\na\chi_0|^2 + \sum_{i,l=1}^n|\pa_i\pa_l\chi_0|^2)|u|^2 + |\na\chi_0|^2|\na u|^2 \Big)e^{2s\vp_2} \dx\dt
\end{align}
for all $\la\ge \la_0$, $s\ge s_0$ and
\begin{align}
\nonumber
&\la \int_Q \chi_0^2\bigg\{ (s\la\vp_2)^{\tau_j -1}|\pa_t u_j|^2 + (s\la\vp_2)^{\tau_j +1}|\na u_j|^2 + (s\la\vp_2)^{\tau_j +3}|u_j|^2 \bigg\}e^{2s\vp_2}\dx\dt \\
\nonumber
&\hspace{0cm} \le C\int_Q \chi_0^2(s\la\vp_2)^{\tau_j}(|D_t^{\f{j}{2k+1}} F|^2 + |u_{j+m}|^2)e^{2s\vp_2}\dx\dt + C(\la)s^{\tau_j}e^{C(\la)s}\int_{\pa Q} (|\na_{x,t}(\chi_0 u_j)|^2 + |\chi_0 u_j|^2)\d S\dt \\
\label{CE:eq6}
&\hspace{0cm} + C\int_Q (s\la\vp_2)^{\tau_j}( |\wt L(\chi_0 u_j) - \chi_0 \wt L(u_j)|^2 + |\pa_t\chi_0|^2|u_j|^2 + s\la^2\vp_2|\na\chi_0|^2|u_j|^2 )e^{2s\vp_2} \dx\dt
\end{align}
for all $\la\ge \la_j$, $s\ge s_j$ and all $j=-k,...,-1,1,...,k$. 
Combining the Carleman inequalities \eqref{CE:eq5}--\eqref{CE:eq6}, i.e., taking the summation, we obtain 
\begin{align*}
&\la \int_Q \chi_0^2\bigg\{(s\la\vp_2)^{\tau_0 -1}\sum_{i,l=1}^n|\pa_i\pa_l u|^2 +(s\la\vp_2)^{\tau_0 +1}|\na u|^2+ \sum_{j=-k}^k\left( (s\la\vp_2)^{\tau_j -1}|\pa_t u_j|^2 + (s\la\vp_2)^{\tau_j +3}|u_j|^2\right) \bigg\}e^{2s\vp_2}\dx\dt \\
&\le C\int_Q \chi_0^2\sum_{j=-k}^k (s\la\vp_2)^{\tau_j}|D_t^{\f{j}{2k+1}} F|^2e^{2s\vp_2}\dx\dt + C\int_Q \chi_0^2\sum_{j=-k}^k (s\la\vp_2)^{\tau_j}|u_{j+m}|^2e^{2s\vp_2}\dx\dt + Low1 + Bdy1
\end{align*}
for all $\la\ge \wt \la:=\max\{\la_j: -k\le j\le k\}$ and $s\ge \wt s:=\max\{s_j: -k\le j\le k\}$. Here $Low1$ and $Bdy1$ are the lower order terms and boundary terms determined as 
\begin{align*}
&Low1 
= C\int_Q \sum_{j=-k}^k (s\la\vp_2)^{\tau_j}\bigg(|\pa_t\chi_0|^2 + |\na\chi_0|^2 + \sum_{i,l=1}^n |\pa_i\pa_l \chi_0|^2\bigg) \left(|\na u_j|^2 + s\la^2\vp_2|u_j|^2 \right)\e^{2s\vp_2} \dx\dt, \\
&Bdy1= C(\la)\e^{C(\la)s}\int_{\pa Q} \sum_{j=-k}^k s^{\tau_j}\left(|\pa_t\chi_0|^2 + |\na\chi_0|^2 + |\chi_0|^2\right)\left(|\na_{x,t} u_j|^2 + |u_j|^2 \right) \d S\dt.
\end{align*}
Now we fix $\tau_j=-\f{4}{2k+1}(j+k)\le 0$, $j=-k,...,k$. Along with the relation 
$$
u_{j+2k+1} = D_t^{\f{j+2k+1}{2k+1}} u = \pa_t (D_t^{\f{j}{2k+1}}u) = \pa_t u_j
$$
which follows after the zero initial condition, a direct calculation yields that
\begin{align*}
&\sum_{j=-k}^k (s\la\vp_2)^{\tau_j -1}|u_{j+2k+1}|^2 = \sum_{j=-k}^k (s\la\vp_2)^{-\f{4j+6k+1}{2k+1}}|u_{j+2k+1}|^2 = \sum_{j=k+1}^{3k+1} (s\la\vp_2)^{-\f{4j-2k-3}{2k+1}}|u_{j}|^2, \\
&\sum_{j=-k}^k (s\la\vp_2)^{\tau_j +3}|u_{j}|^2 = \sum_{j=-k}^k (s\la\vp_2)^{-\f{4j-2k-3}{2k+1}}|u_{j}|^2, \\
&\sum_{j=-k}^k (s\la\vp_2)^{\tau_j}|u_{j+m}|^2 = \sum_{j=-k}^k (s\la\vp_2)^{-\f{4j+4k}{2k+1}}|u_{j+m}|^2 = \sum_{j=-k+m}^{k+m} (s\la\vp_2)^{-\f{4j+4k-4m}{2k+1}}|u_{j}|^2,
\end{align*} 
which imply that
\begin{align}
\nonumber
&\la \int_Q \chi_0^2\bigg\{\sum_{i,l=1}^n(s\la\vp_2)^{-\f{6k+1}{2k+1}}|\pa_i\pa_l u|^2 +(s\la\vp_2)^{-\f{2k-1}{2k+1}}|\na u|^2+ \sum_{j=-k}^{3k+1}(s\la\vp_2)^{-\f{4j-2k-3}{2k+1}}|u_{j}|^2\bigg\}e^{2s\vp_2}\dx\dt \\
\nonumber
&\le C\int_Q \chi_0^2\sum_{j=-k}^k (s\la\vp_2)^{-\f{4j+4k}{2k+1}}|D_t^{\f{j}{2k+1}} F|^2e^{2s\vp_2}\dx\dt + C\int_Q \chi_0^2\sum_{j=-k+m}^{k+m} (s\la\vp_2)^{-\f{4j+4k-4m}{2k+1}}|u_j|^2e^{2s\vp_2}\dx\dt \\
\label{CE:eq7}
&\hspace{1cm}+ Low1 + Bdy1
\end{align}
for all $\la\ge \wt\la$ and all $s\ge \wt s$. Since $\al = \f{m}{2k+1}\in(0,\f{3}{4}]$ implies 
$$
-\f{4j-2k-3}{2k+1} \ge -\f{4j+4k-4m}{2k+1}\quad \text{for } j=-k+m,...,k+m, 
$$
we then fix $\la$ large so that the second terms on the RHS of \eqref{CE:eq7} can be absorbed into the LHS of it ($\la\ge \hat{\la}$). Moreover, we have
\begin{align*}
&Low1 \le C\int_Q \sum_{j=-k}^k \bigg(|\pa_t\chi_0|^2 + |\na\chi_0|^2 + \sum_{i,l=1}^n |\pa_i\pa_l \chi_0|^2\bigg) \left(|\na u_j|^2 + s\vp_2|u_j|^2 \right)\e^{2s\vp_2} \dx\dt \le Low, \\
&Bdy1 \le C\e^{Cs}\int_{\pa Q} \sum_{j=-k}^k \left(|\pa_t\chi_0|^2 + |\na\chi_0|^2 + |\chi_0|^2\right)\left(|\na_{x,t} u_j|^2 + |u_j|^2 \right) \d S\dt = Bdy
\end{align*}
since $s,\la,\vp_2\ge 1$ and $\tau_j\le 0$. Therefore, \eqref{CE:eq7} yields
\begin{align}
\nonumber
&\int_Q \chi_0^2\bigg\{\sum_{i,l=1}^n(s\vp_2)^{-\f{6k+1}{2k+1}}|\pa_i\pa_l u|^2 +(s\vp_2)^{-\f{2k-1}{2k+1}}|\na u|^2+ \sum_{j=-k}^{3k+1}(s\vp_2)^{-\f{4j-2k-3}{2k+1}}|D_t^{\f{j}{2k+1}}u|^2\bigg\}e^{2s\vp_2}\dx\dt \\
\label{CE:eq8}
&\le C\int_Q \chi_0^2\sum_{j=-k}^k (s\vp_2)^{-\f{4j+4k}{2k+1}}|D_t^{\f{j}{2k+1}} F|^2e^{2s\vp_2}\dx\dt + Low + Bdy
\end{align}
for $\la$ large fixed and all $s\ge \wt s$.

For the case (\romannumeral2), we repeat the strategy of above to obtain the equations similar to \eqref{CE:eq3}:
\begin{equation}
\label{CE:eq9}
\wt L(\wt{u_j}) = \chi_0 D_t^{\f{j}{2k}}F - (\chi_0 q) u_{j+m} - 2\na\chi_0\cdot\na u_j  + (B\cdot\na\chi_0 - \Delta\chi_0 + \pa_t\chi_0) u_j
\end{equation}
for all $j=-k+1,...,k$ where we denote $u_j:= D_t^{\f{j}{2k}} u$ and $\wt{u_j} = \chi_0 u_j$ for all $j\in \Z$ instead. Again we apply Carleman estimate for parabolic equations (Lemma \ref{CE:para}) to \eqref{CE:eq8} and take the summation over $j$ from $-k+1$ to $k$ which reads
\begin{align}
\nonumber
&\la \int_Q \chi_0^2\bigg\{(s\la\vp_2)^{\tau_0 -1}\sum_{i,l=1}^n|\pa_i\pa_l u|^2 +(s\la\vp_2)^{\tau_0 +1}|\na u|^2+ \sum_{j=-k+1}^k\left( (s\la\vp_2)^{\tau_j -1}|u_{j+2k}|^2 + (s\la\vp_2)^{\tau_j +3}|u_j|^2\right) \bigg\}e^{2s\vp_2}\dx\dt \\
\label{CE:eq10}
&\le C\int_Q \chi_0^2\sum_{j=-k+1}^k (s\la\vp_2)^{\tau_j}|D_t^{\f{j}{2k}} F|^2e^{2s\vp_2}\dx\dt + C\int_Q \chi_0^2\sum_{j=-k+1}^k (s\la\vp_2)^{\tau_j}|u_{j+m}|^2e^{2s\vp_2}\dx\dt + Low2 + Bdy2
\end{align}
for all $\la\ge \wt \la$ and $s\ge \wt s$. Here $Low2$ and $Bdy2$ are the lower order terms and boundary terms determined as 
\begin{align*}
&Low2 = C\int_Q \sum_{j=-k+1}^k (s\la\vp_2)^{\tau_j}\bigg(|\pa_t\chi_0|^2 + |\na\chi_0|^2 + \sum_{i,l=1}^n |\pa_i\pa_l \chi_0|^2\bigg) \left(|\na u_j|^2 + s\la^2\vp_2|u_j|^2 \right)\e^{2s\vp_2} \dx\dt, \\
&Bdy2= C(\la)\e^{C(\la)s}\int_{\pa Q} \sum_{j=-k+1}^k s^{\tau_j}\left(|\pa_t\chi_0|^2 + |\na\chi_0|^2 + |\chi_0|^2\right)\left(|\na_{x,t} u_j|^2 + |u_j|^2 \right) \d S\dt.
\end{align*}
For the case (\romannumeral2), we choose $\tau_j=-\f{2}{k}(j+k-1)\le 0$, $j=-k+1,...,k$. By direct calculation , we obtain
\begin{align*}
&\sum_{j=-k+1}^k (s\la\vp_2)^{\tau_j -1}|u_{j+2k}|^2 = \sum_{j=-k+1}^k (s\la\vp_2)^{-\f{2j+3k-2}{k}}|u_{j+2k}|^2 = \sum_{j=k+1}^{3k} (s\la\vp_2)^{-\f{2j-k-2}{k}}|u_{j}|^2, \\
&\sum_{j=-k+1}^k (s\la\vp_2)^{\tau_j +3}|u_{j}|^2 = \sum_{j=-k+1}^k (s\la\vp_2)^{-\f{2j-k-2}{k}}|u_{j}|^2, \\
&\sum_{j=-k+1}^k (s\la\vp_2)^{\tau_j}|u_{j+m}|^2 = \sum_{j=-k+1}^k (s\la\vp_2)^{-\f{2j+2k-2}{k}}|u_{j+m}|^2 = \sum_{j=-k+m+1}^{k+m} (s\la\vp_2)^{-\f{2j+2k-2m-2}{k}}|u_{j}|^2.
\end{align*} 
Since $\al=\f{m}{2k}\in (0,\f 3 4]$ implies that
$$
-\f{2j-k-2}{k} \ge -\f{2j+2k-2m-2}{k} \quad for\ j=-k+m+1,...,k+m,
$$
we fix $\la$ so large that the second terms on the RHS of \eqref{CE:eq10} can be absorbed. Therefore, we have the following Carleman inequality
\begin{align}
\nonumber
&\int_Q \chi_0^2\bigg\{\sum_{i,l=1}^n(s\vp_2)^{-\f{3k-2}{k}}|\pa_i\pa_l u|^2 +(s\vp_2)^{-\f{k-2}{k}}|\na u|^2+ \sum_{j=-k}^{3k}(s\vp_2)^{-\f{2j-k-2}{k}}|D_t^{\f{j}{2k}}u|^2\bigg\}e^{2s\vp_2}\dx\dt \\
\label{CE:eq11}
&\le C\int_Q \chi_0^2\sum_{j=-k+1}^k (s\vp_2)^{-\f{2j+2k-2}{k}}|D_t^{\f{j}{2k}} F|^2e^{2s\vp_2}\dx\dt + Low + Bdy
\end{align}
for $\la$ large fixed and all $s\ge \wt s$.

Theorem \ref{thm-CE<=0.75} now follows after the two Carleman inequlaities \eqref{CE:eq8} and \eqref{CE:eq11}.
\end{proof}
\begin{rem}
Under the assumption $h(0) = 0$, we note that 
\begin{equation}
\label{esti-alpha<beta}
\|D_t^{\al}h\|_{L^2(0,T)} \le C\|D_t^\be h\|_{L^2(0,T)}, \quad -1<\al<\be<2,
\end{equation}
which actually can be derived by using the definition of $D_t^\al$ and Young's inequality. In view of this estimate, we can also rewrite $Bdy$ as
$$
Bdy = C_1e^{Cs}\int_{\pa Q} \left(|\pa_t\chi_0|^2 + |\na\chi_0|^2 + |\chi_0|^2\right)(|\na_{x,t} D_t^{\f12} u|^2 + |D_t^{\f12} u|^2)\d \Sg.
$$
However, constant $C_1$ now depends on $k$, which is the main difficulty in dealing with the equation \eqref{equ-gov2} with the time-fractional derivative of irrational order by using the density of rational numbers in $\R$. The Carleman estimate for the equation \eqref{equ-gov2} with the general order time-fractional derivatives remains open.
\end{rem}


\subsection{Application to a lateral Cauchy problem for the sup-diffusion}
\label{sec-LCP=0.75}
In this subsection, we employ the Carleman estimate in Theorem \ref{thm-CE<=0.75} to investigate the conditional stability for the lateral Cauchy problem. For this, we recall the partial differential equation:
\begin{equation}
\label{sy:0.75n}
\pa_t u + {q(x) \pa_t^{\f34} }u - \Delta u + B\cdot\na u + {c}u = F \quad \mbox{in }Q 
\end{equation}
with the zero initial condition:
\begin{equation}
\label{con:ini}
u(x,0) = 0, \quad x\in\Om.
\end{equation}

We mainly follow the steps in \cite{Y09} Theorem 5.1. Instead of a classical Carleman estimate for parabolic equation, we should apply our Carleman estimate established in Section \ref{sec-CE<=0.75}. 

\begin{proof}[\bf Proof of Theorem \ref{thm-LCP=0.75}]
By the choice of $\Om_0$, we have $\ov{\Om_0}\subset\Om_1$ where $\Om_1$ is defined in Section \ref{sec-intro}. Then we can find a sufficiently large $N > 1$ such that
\begin{equation}
\label{con:N}
\left\{x\in \Om_1: d(x) > \f3N\|d\|_{C(\ov{\Om_1})} \right\} \cap\ov\Om \supset\Om_0.
\end{equation}
Moreover we choose $\be > 0$ such that
\begin{equation}
\label{con:beta}
\be\ep^2\le \|d\|_{C(\ov{\Om_1})}\le 2\be\ep^2.
\end{equation}
For any $t_0\in [\sqrt2\ep, T - \sqrt2\ep]$, we set $\mu_k := \exp\{\f{k}N\|d\|_{C(\ov{\Om_1})} - \f{\be\ep^2}N + c_0\}$, $E_k := \{(x,t)\in Q_1: \vp_2(x,t) > \mu_k \}$ and $D_k := E_k\cap Q$, $k = 1,2,3$. Recall that $c_0$ is the same constant as that in \eqref{con:weight}. 
Then we can easily verify the following two facts:
\vspace{0.2cm}

\noindent (\romannumeral1)  $\Om_0\times\big(t_0 - \f{\ep}{\sqrt{N}}, t_0 + \f{\ep}{\sqrt{N}}\big) \subset D_3 \subset D_2 \subset D_1 \subset \ov\Om\times (t_0 - \sqrt{2}\ep, t_0 + \sqrt{2}\ep)$, 
\vspace{0.2cm}

\noindent (\romannumeral2)  $\pa D_1 \subset \Sigma_1\cup\Sigma_2, \quad \Sigma_1\subset \Ga\times (0,T), \ \Sigma_2 = \{(x,t)\in Q: \vp_2(x,t) = \mu_1\}$.
\vspace{0.2cm}

Next, we specify the cut-off function $\chi_0$ in Theorem \ref{thm-CE<=0.75} with $D_0 = E_2$ and $D = E_1$. In detail, $\chi_0\in C^\infty(\R^{n+1})$ such that $0\le\chi_0\le 1$ and
$$
\chi_0(x,t) =
\begin{cases}
\ 1, \quad \vp_2(x,t) > \mu_2, \\
\ 0, \quad \vp_2(x,t) \le \mu_1.
\end{cases}
$$
Thus from Theorem \ref{thm-CE<=0.75} (denominator $k=4$ in this proof), it follows that
\begin{equation}
\label{LCP:eq1}
\begin{aligned}
&\int_{D_1} \chi_0^2\Big(\sum_{j=-1}^6 s^{2-j}|D_t^{\f{j}4} u|^2 + |\na u|^2 + s^{-2}\sum_{i,j=1}^n|\pa_i\pa_j u|^2 \Big)e^{2s\vp_2} \dx\dt \\
&\hspace{3cm} \le C\int_{D_1}\chi_0^2 \Big(\sum_{j=-1}^2 |D_t^{\f{j}4} F|^2 \Big)e^{2s\vp_2}\dx\dt + Low + Bdy
\end{aligned}
\end{equation}
for all $s\ge s_0$ where $Low$ and $Bdy$ is the same notations as \eqref{Low}-\eqref{Bdy}. In the above inequality, we have used the fact that $\vp_2$ is upper and lower bounded by a constant which is independent of $s$. Since the second term $Low$ on the RHS includes some derivates of $\chi_0$, it vanishes in $\ov{E_2}$ and outside of $E_1$, which implies that $\vp_2(x,t) \le \mu_2$ in $E_1\setminus \ov{E_2}$ ($\supset D_1\setminus\ov{D_2}$) and thus reads
\begin{align*}
&Low = Cs\int_Q \left(|\pa_t\chi_0|^2 + |\na\chi_0|^2 + \sum_{i,j=1}^n |\pa_i\pa_j \chi_0|^2\right)\sum_{j=-1}^{2} \left(|\na(D_t^{\f{j}4} u)|^2 + |D_t^{\f{j}4} u|^2 \right)\e^{2s\vp_2} \dx\dt \\
&\hspace{0.68cm} \le Cs\e^{2s\mu_2}\int_{D_1\setminus\ov{D_2}} \sum_{j=-1}^2 \left(|\na(D_t^{\f{j}4} u)|^2 + |D_t^{\f{j}4} u|^2 \right)\dx\dt \le Cs\e^{2s\mu_2}\|D_t^{\f12}u\|_{H^{1,0}(Q)}^2.
\end{align*}
The last inequality above holds after the estimate \eqref{esti-alpha<beta}. Furthermore, fact (\romannumeral2) and the choice of $\chi_0$ yield that $\na\chi_0,\pa_t\chi_0$ and $\chi_0$ vanish on both $\Om\times\{0,T\}$ and $(\pa\Om\setminus\Ga)\times(0,T)$ which indicate that
\begin{align*}
&Bdy = C\e^{Cs}\int_{\pa Q} \left(|\pa_t\chi_0|^2 + |\na\chi_0|^2 + |\chi_0|^2\right)\sum_{j=-1}^{2} \left(|\na_{x,t} (D_t^{\f{j}{4}} u)|^2 + |D_t^{\f{j}{4}} u|^2 \right) \d S\dt \\
&\hspace{0.68cm} \le C\e^{Cs}\int_{\Ga\times(0,T)}\sum_{j=-1}^{2} \left(|\na_{x,t} (D_t^{\f{j}{4}} u)|^2 + |D_t^{\f{j}{4}} u|^2 \right) \d S\dt \\
&\hspace{0.68cm} \le C\e^{Cs}\left(\|\na_{x,t}(D_t^{\f12}u)\|_{L^2(\Ga\times(0,T))}^2 + \|D_t^{\f12}u\|_{L^2(\Ga\times(0,T))}^2\right).
\end{align*}
The last inequality again comes from \eqref{esti-alpha<beta}. Thus, we obtain
\begin{align*}
&RHS\ of\ \eqref{LCP:eq1} \le C\e^{Cs}\sum_{j=-1}^2 \int_{D_1}|D_t^{\f{j}4}F|^2\dx\dt + Cs\e^{2s\mu_2}\|D_t^{\f12}u\|_{H^{1,0}(Q)}^2 \\
&\hspace{2.3cm} + C\e^{Cs}\left(\|\na_{x,t}(D_t^{\f12}u)\|_{L^2(\Ga\times(0,T))}^2 + \|D_t^{\f12}u\|_{L^2(\Ga\times(0,T))}^2\right) \\
&\hspace{2cm} \le Cs\e^{2s\mu_2}M^2 + C\e^{Cs}\cD^2
\end{align*}
where $M$ and $\cD$ are defined in Theorem \ref{thm-LCP=0.75}.
On the other hand, the inclusion $\Om_0\times\big(t_0 - \f{\ep}{\sqrt{N}}, t_0 + \f{\ep}{\sqrt{N}}\big) \subset D_3 \subset D_2$ and $\chi_0 = 1$ in $D_2$ imply that
\begin{align*}
&LHS\ of\ \eqref{LCP:eq1} = \int_{D_2} \Big(\sum_{j=-1}^6 s^{2-j}|D_t^{\f{j}4} u|^2 + |\na u|^2 + s^{-2}\sum_{i,j=1}^n|\pa_i\pa_j u|^2 \Big)e^{2s\vp_2} \dx\dt  \\
&\hspace{2cm} \ge \int_{D_3} \Big(\sum_{j=-1}^6 s^{2-j}|D_t^{\f{j}4} u|^2 + |\na u|^2 + s^{-2}\sum_{i,j=1}^n|\pa_i\pa_j u|^2 \Big)e^{2s\vp_2} \dx\dt \\
&\hspace{2cm} \ge e^{2s\mu_3}\int_{t_0 - \f{\ep}{\sqrt{N}}}^{t_0 + \f{\ep}{\sqrt{N}}} \int_{\Om_0} \Big(\sum_{j=-1}^6 s^{2-j}|D_t^{\f{j}4} u|^2 + |\na u|^2 + s^{-2}\sum_{i,j=1}^n|\pa_i\pa_j u|^2 \Big) \dx\dt.
\end{align*}
Therefore \eqref{LCP:eq1} yields 
$$
\|u\|_{H^{2,1}(\Om_0\times (t_0-\f{\ep}{\sqrt{N}}, t_0+\f{\ep}{\sqrt{N}}))}^2 \le Cs^3e^{-2s(\mu_3-\mu_2)} M^2 + Cs^2e^{Cs}\cD^2
$$
for all $s\ge s_0$. Since $\sup_{s>0} s^3e^{-s(\mu_3-\mu_2)} < \infty$ and $s^2 \le e^{Cs}$ for $s$ large enough (e.g. $s\ge s_1$), we have
$$
\|u\|_{H^{2,1}(\Om_0\times (t_0-\f{\ep}{\sqrt{N}}, t_0+\f{\ep}{\sqrt{N}}))}^2 \le Ce^{-C_0s}M^2 + Ce^{Cs}\cD^2
$$
for all $s\ge s_2:=\max\{s_0,s_1\}$, $C_0 := \mu_3 - \mu_2 > 0$. After the change of the variable $s-s_2\to s$, we obtain
\begin{equation}
\label{LCP:eq2}
\|u\|_{H^{2,1}(\Om_0\times (t_0-\f{\ep}{\sqrt{N}}, t_0+\f{\ep}{\sqrt{N}}))}^2 \le Ce^{-C_0s}M^2 + Ce^{Cs}\cD^2
\end{equation}
for all $s\ge 0$. The generic constant $C$ depends on $s_2$, $\ep$ and the choice of $\Om_0$, etc., but is still independent of $s$.

Let $m\in \N$ satisfy
$$
\sqrt2\ep + \f{m\ep}{\sqrt{N}} \le T - \sqrt{2}\ep < \sqrt{2}\ep + \f{(m+1)\ep}{\sqrt{N}} < T.
$$
We here notice that the constant $C$ is also independent of $t_0$ provided that $t_0\in [\sqrt{2}\ep, T - \sqrt{2}\ep]$. Thus, by taking $t_0 = \sqrt{2}\ep + \f{k\ep}{\sqrt{N}}$, $k = 0,1,2,..., m$ in \eqref{LCP:eq2}, summing up over $k$ and replacing $\sqrt{2}\ep$ by $\ep$, we obtain
\begin{equation}
\label{LCP:eq3}
\|u\|_{H^{2,1}(\Om_0\times (\ep, T-\ep))}^2 \le Ce^{-C_0s}M^2 + Ce^{Cs}\cD^2
\end{equation}
for all $s\ge 0$.  

Finally, we show the estimate of H\"{o}lder type.

Firstly, let $\cD = 0$. Then letting $s\rightarrow \infty$ in \eqref{LCP:eq3}, we see that $u = 0$ in $\Om_0\times (\ep, T-\ep)$. So that the conclusion holds true.

Secondly, let $\cD \neq 0$. We divide it into two cases. In the case of $\cD \ge M$, we see that \eqref{LCP:eq3} implies $\|u\|_{H^{2,1}(\Om_0\times (\ep, T-\ep))} \le Ce^{Cs}\cD$ for all $s\ge 0$. This has already proved the theorem. 
Now if $\cD \le M$. We choose $s > 0$ minimizing the right-hand side of \eqref{LCP:eq3}, that is, 
$$
e^{-C_0s}M^2 = e^{Cs}\cD^2,
$$
which yields
$$
s = \f{2}{C + C_0} \log{\f{M}{\cD}}
$$
in view of that $\cD\ne0$. Then \eqref{LCP:eq3} gives
$$
\|u\|_{H^{2,1}(\Om_0\times (\ep, T-\ep))}^2 \le 2CM^{\f{C}{C+C_0}}\cD^{\f{C_0}{C+C_0}}.
$$
The proof of Theorem \ref{thm-LCP=0.75} is completed. 
\end{proof}
\subsection{Application to an inverse source problem for the sup-diffusion}
\label{sec:ISP}
In this subsection, we consider another application of Theorem \ref{thm-CE<=0.75}, say, the stability for the inverse source problem (Problem \ref{prob-ISP=0.75}). We point out here that the uniqueness result for this kind of inverse problem can be proved by a similar argument in \cite{JLLY16}.

Before giving the proof of Theorem \ref{thm-ISP=0.75}, we first recall the notation of the bounded domain $\Om_1$ defined in \eqref{con:Om1}, and function $d\in C^2(\ov{\Om_1})$ can be chosen so that \eqref{con:d} holds. Next for any $\ep>0$, we {recall the} level set of $\ep$ related to function $d$ as follows
$$
\Om_\ep := \{x\in\Om: d(x) > \ep\}.
$$
Note that $\Om_\ep \neq \emptyset$ for small $\ep$ and $\ov{\Om_\ep}\cap \pa\Om \subset \Gamma$.
Now we are ready to give a proof of the stability result. 

\begin{proof}[\bf Proof of Theorem \ref{thm-ISP=0.75}]
Recall that our weight function is chosen as
$$
\vp_2(x,t) = e^{\la\psi_2(x,t)}, \quad \psi_2(x,t) = d(x) - \beta(t-t_0)^2 + c_0.
$$
Here, $c_0$ is a constant such that $\psi_2$ is always nonnegative and $s,\la$ are two large parameters while the parameter $\beta > 0$ will be fixed later. Then on both sides of our governing equation \eqref{sy:ISP}, we integrate over the domain $\Om_{3\ep}$ at time $t = t_0$ and obtain
\begin{equation}
\label{ISP:eq1}
\begin{aligned}
&\int_{\Om_{3\ep}} |R(x,t_0)f(x)|^2 e^{2s\vp_2(x,t_0)} \dx 
\le \int_{\Om_{3\ep}} |\pa_t u(x,t_0)|^2\e^{2s\vp_2(x,t_0)} \dx 
+ {C}\int_{\Om_{3\ep}} \left|D_t^{\f34} u(x,t_0)\right|^2\e^{2s\vp_2(x,t_0)} \dx \\
&\hspace{3cm} + \int_{\Om_{3\ep}} |-\De u(x,t_0) + B(x)\cdot\na u(x,t_0) + {c}(x)u(x,t_0)|^2\e^{2s\vp_2(x,t_0)}\dx.
\end{aligned}
\end{equation}
Easily, we see that the third term on the RHS is bounded by $Ce^{Cs}\|u(\cdot,t_0)\|_{H^2(\Om_{3\ep})}^2$ and the LHS can be estimated with some $C_0>0$:
\begin{equation}
\label{ISP:eq2}
\int_{\Om_{3\ep}} |R(x,t_0)f(x)|^2e^{2s\vp_2(x,t_0)}dx \ge C_0 \int_{\Om_{3\ep}} |f(x)|^2e^{2s\vp_2(x,t_0)}dx
\end{equation}
under the first assumption in \eqref{con:R}.  The key point is how to estimate the first and second terms on the RHS of \eqref{ISP:eq1}. 

By our notation $\Om_{\ep} = \{x\in\Om: d(x) > \ep\}$ for any $\ep>0$, we further set 
$$
Q_\ep := \{(x,t)\in Q: \psi_2(x,t) > \ep + c_0\}, \quad \ep>0.
$$
Then we have the following relations:
\vspace{0.1cm}

(\romannumeral1)  $Q_\ep\subset\Om_{\ep}\times (0,T)$, 
\vspace{0.1cm}

(\romannumeral2)  $Q_\ep\supset\Om_{\ep}\times \{t_0\}$.
\vspace{0.1cm}

\noindent In fact, if $(x,t)\in Q_\ep$, we have $d(x) - \be(t-t_0)^2 > \ep$, i.e. $d(x)>\be(t-t_0)^2 + \ep> \ep$. This means $x\in\Om_{\ep}$. (\romannumeral1) is verified. On the other hand, if $x\in\Om_\ep$ and $t=t_0$ then $\psi_2(x,t) = d(x) - \be(t-t_0)^2 + c_0 = d(x) + c_0 > \ep + c_0$. That is, $(x,t)\in Q_\ep$. (\romannumeral2) is verified.
Furthermore, we choose $\be = \f{\|d\|_{C(\ov{\Om_1})}}{\de^2}$ where $\de:= \min\{t_0, T- t_0\}$ so that 
\vspace{0.1cm}

(\romannumeral3)  $\ov{Q_\ep}\cap (\Om\times\{0,T\}) = \emptyset$
\vspace{0.1cm}

\noindent is valid. Indeed, for $\forall (x,t)\in \Om\times\{0,T\}$, $\psi(x,t) = d(x) - \be(t-t_0)^2 + c_0 \le \|d\|_{C(\ov{\Om_1})} - \be\de^2 + c_0 = c_0$. This leads to $(x,t)\notin \ov{Q_\ep}$. 

Relations (\romannumeral1) -- (\romannumeral3) guarantee that $Q_\ep$ is a sub-domain of $Q$ and $\pa Q_\ep\cap \pa Q\subset \Ga\times (0,T)$. Moreover, we assert that 
\vspace{0.1cm}

(\romannumeral4)   $\Om_{3\ep}\times (t_0-\delta_{\ep},t_0)\subset Q_{2\ep}, \quad \de_\ep := \sqrt{\f{\ep}{\be}} = \sqrt{\f{\ep}{\|d\|}}\de$.
\vspace{0.1cm}

\noindent Actually, for any $(x,t)\in \Om_{3\ep}\times (t_0-\delta_{\ep},t_0)$, we have
$$
\psi(x,t) = d(x) - \beta(t-t_0)^2 + c_0 > 3\ep - \beta\delta_{\ep}^2 + c_0 = 2\ep + c_0.
$$
That is, $(x,t)\in Q_{2\ep}$. 

Now we construct a function {$\eta\in C^2[0,T]$ such that $0\le \eta\le 1$ and
$$
\eta = \left\{
\begin{aligned}
& 1 \quad in\ [t_0-\f12\de_\ep, t_0+\f12\de_\ep], \\
& 0 \quad in\ [0,t_0-\de_\ep]\cup [t_0+ \de_\ep, T]
\end{aligned}
\right. 
$$
}
for any small $\ep<\|d\|_{C(\ov{\Om_1})}$. Then by the use of the condition \eqref{con:initial}, and noting that $\eta(t_0-\de_\ep) = 0$, $\eta(t_0) = 1$, we have
\begin{equation}
\label{ISP:eq3}
\begin{aligned}
&\int_{\Om_{3\ep}} |\pa_t u(x,t_0)|^2e^{2s\vp_2(x,t_0)}\dx = \int_{\Om_{3\ep}} |\eta(t_0)\pa_t u(x,t_0)|^2\e^{2s\vp_2(x,t_0)}\dx \\
&\hspace{4.28cm} = \int_{t_0-\de_\ep}^{t_0} \f{\d}{\dt} \int_{\Om_{3\ep}}|\eta\pa_t u|^2\e^{2s\vp_2}\dx\dt \\
&\hspace{4.28cm} = \int_{t_0-\de_\ep}^{t_0}\int_{\Om_{3\ep}} 2\eta\pa_t u (\eta\pa_t^2 u + \pa_t \eta \pa_t u)\e^{2s\vp_2}\dx\dt \\
&\hspace{4.5cm} -\int_{t_0-\de_\ep}^{t_0}\int_{\Om_{3\ep}} 4(t-t_0)\be s\la\vp_2|\eta\pa_t u|^2\e^{2s\vp_2}\dx\dt \\
&\hspace{4.28cm} \le C \int_{Q_{2\ep}} (s^{-2}|\pa_t^2 u|^2 + s^2 |\pa_t u|^2)\e^{2s\vp_2} \dx\dt
\end{aligned}
\end{equation}
and
\begin{equation}
\label{ISP:eq4}
\begin{aligned}
&\int_{\Om_{3\ep}} |D_t^{\f34} u(x,t_0)|^2e^{2s\vp_2(x,t_0)}\dx = \int_{\Om_{3\ep}} |\eta(t_0)D_t^{\f34} u(x,t_0)|^2e^{2s\vp_2(x,t_0)}\dx \\
&\hspace{4.28cm} = \int_{t_0-\de_\ep}^{t_0} \f{\d}{\dt} \int_{\Om_{3\ep}}|\eta D_t^{\f34} u|^2\e^{2s\vp_2}\dx\dt \\
&\hspace{4.28cm} = \int_{t_0-\de_\ep}^{t_0}\int_{\Om_{3\ep}} 2\eta D_t^{\f34} u (\eta\pa_t(D_t^{\f34} u) + \pa_t \eta D_t^{\f34} u)\e^{2s\vp_2}\dx\dt \\
&\hspace{4.5cm} -\int_{t_0-\de_\ep}^{t_0}\int_{\Om_{3\ep}} 4\be s\la\vp_2(t-t_0)|\eta D_t^{\f34} u|^2\e^{2s\vp_2}\dx\dt \\
&\hspace{4.28cm} \le \f{C}{s} \int_{Q_{2\ep}} (s^{-1}|D_t^{\f74} u|^2 + s^3 |D_t^{\f34} u|^2)\e^{2s\vp_2}\dx\dt.
\end{aligned}
\end{equation}

Now we intend to use the Carleman estimate established in Section \ref{sec-CE<=0.75} to evaluate the RHS of \eqref{ISP:eq3} and \eqref{ISP:eq4}.  For this, we take time derivative on both sides of \eqref{sy:ISP} and then $u_1 := \pa_t u$ reads the following equation
\begin{equation}
\label{sy:ISP1}
\pa_t u_1 + {q(x)\pa_t^{\f34}} u_1 - \De u_1 + B\cdot\na u_1 + {c} u_1 = (\pa_t R)f, \quad \mbox{in } Q.
\end{equation}
Noting the compatible condition $R(\cdot,0)=0$, we apply Theorem \ref{thm-CE<=0.75} to \eqref{sy:ISP1}, and then we have
$$
\begin{aligned}
&\int_Q \chi_0^2(s^{-2}|\pa_t u_1|^2 + s^{-1}|D_t^{\f34} u_1|^2 + s^2|u_1|^2 + s^3|D_t^{-\f14} u_1|^2)e^{2s\vp_2}\dx\dt \\
&\hspace{3.5cm} \le C\int_Q \chi_0^2\bigg(\sum_{j=3}^6\Big|D_t^{\f{j}4} R\Big|^2\bigg)|f|^2e^{2s\vp_2}\dx\dt + Low1 + Bdy1
\end{aligned}
$$
for all $s\ge s_1\ge 1$, where the terms $Low1$ and $Bdy1$ are defined as
\begin{align*}
&Low1 = Cs\int_Q \left(|\pa_t\chi_0|^2 + |\na\chi_0|^2 + \sum_{i,j=1}^n|\pa_i\pa_j \chi_0|^2\right)\bigg(\sum_{j=-1}^2 \big(|\na(D_t^{\f{j}4 } u_1)|^2 + |D_t^{\f{j}4} u_1|^2 \big)\bigg)e^{2s\vp_2} \dx\dt, \\
&{Bdy1 = Ce^{Cs}\int_{\pa Q} \left(|\pa_t\chi_0|^2 + |\na\chi_0|^2 + |\chi_0|^2\right)\sum_{j=-1}^2\big(|\na_{x,t} (D_t^{\f{j}4} u_1)|^2 + |D_t^{\f{j}4} u_1|^2 \big) \d S\dt}.
\end{align*}
We choose the cut-off function $\chi_0\in C^\infty(\R^{n+1})$ such that $0\le \chi_0\le 1$ and
$$
\chi_0 = \left\{
\begin{aligned}
& 1, \qquad \psi_2(x,t) > 2\ep + c_0, \\
& 0, \qquad \psi_2(x,t) \le \ep + c_0.
\end{aligned}
\right.
$$
Therefore $\chi_0 = 1$ in $Q_{2\ep}$ while its derivatives vanish in $Q_{2\ep}$ which enable us to rewrite the above Carleman inequality as follows:
\begin{equation}
\label{ISP:eq5}
\begin{aligned}
&\int_{Q_{2\ep}} (s^{-2}|\pa_t^2 u|^2 + s^{-1}|D_t^{\f74} u|^2 + s^2|\pa_t u|^2 + s^3|D_t^{\f34} u|^2)e^{2s\vp_2}\dx\dt \\
&\hspace{3cm} \le C\int_{Q_{\ep}} \bigg(\sum_{j=3}^6\Big|D_t^{\f{j}4} R\Big|^2\bigg)|f|^2e^{2s\vp_2}\dx\dt + Low2 + Bdy2
\end{aligned}
\end{equation}
for all $s\ge s_1\ge 1$, where 
\begin{align*}
&Low2 = Cse^{2se^{\la(2\ep + c_0)}}\int_{Q_\ep\setminus Q_{2\ep}} \big(|\na(D_t^{\f{3}2} u)|^2 + |D_t^{\f{3}2} u|^2 \big)\dx\dt, \\
&Bdy2 = Ce^{Cs}\int_{\Ga\times(0,T)} \big(|\na_{x,t} (D_t^{\f{3}2} u)|^2 + |D_t^{\f{3}2} u|^2 \big) \d S\dt.
\end{align*}
The above inequality is derived from the choice of $\chi_0$ and the estimate \eqref{esti-alpha<beta} as well as the relations (\romannumeral1)-(\romannumeral3) which indicate that both the derivatives of $\chi_0$ and $\chi_0$ itself vanish on $\pa Q$ except for the lateral boundary $\Ga\times(0,T)$.
Now substituting \eqref{ISP:eq5} into \eqref{ISP:eq3} and \eqref{ISP:eq4} implies that
$$
\begin{aligned}
&\int_{\Om_{3\ep}} |\pa_t u(x,t_0)|^2e^{2s\vp_2(x,t_0)} \dx + \int_{\Om_{3\ep}} \Big|D_t^{\f34} u(x,t_0)\Big|^2e^{2s\vp_2(x,t_0)} \dx \\
&\hspace{3cm} \le C\int_{Q_{\ep}} \bigg(\sum_{j=3}^6\Big|D_t^{\f{j}4} R\Big|^2\bigg)|f|^2e^{2s\vp_2}\dx\dt + Low2 + Bdy2.
\end{aligned}
$$
Combined with \eqref{ISP:eq1} and \eqref{ISP:eq2}, we obtain
\begin{align*}
\int_{\Om_{3\ep}} |f(x)|^2\e^{2s\vp_2(x,t_0)} \dx \le C\e^{Cs}\|u(\cdot,t_0)\|_{H^2(\Om_{3\ep})}^2 + C\int_{Q_{\ep}} |f|^2\e^{2s\vp_2}\dx\dt  + Low2 + Bdy2.
\end{align*}
Moreover, we divide the second term on the RHS into two parts:
\begin{align*}
C\int_{Q_\ep} |f|^2\e^{2s\vp_2}\dx\dt &= C\int_{Q_{3\ep}} |f|^2\e^{2s\vp_2}\dx\dt + C\int_{Q_{\ep}\setminus Q_{3\ep}} |f|^2 \e^{2s\vp_2} \dx\dt \\
&\le C\int_{Q_{3\ep}} |f|^2\e^{2s\vp_2}\dx\dt + C\e^{2s\e^{\la(3\ep + c_0)}}\int_{Q_{\ep}\setminus Q_{3\ep}} |f|^2 \dx\dt,
\end{align*}
which leads to
\begin{align*}
&\int_{\Om_{3\ep}} |f(x)|^2 \e^{2s\vp_2(x,t_0)} \dx \le C\int_{Q_{3\ep}} |f|^2 \e^{2s\vp_2}\dx\dt + C\e^{2s \e^{\la(3\ep + c_0)}}\int_{Q_{\ep}\setminus Q_{3\ep}} |f|^2 \dx\dt \\
&\hspace{3.5cm} + C\e^{2s\e^{\la(2\ep + c_0)}}\int_{Q_\ep} {\big(|\na(D_t^{\f32} u)|^2 + |D_t^{\f32} u|^2 \big)}\dx\dt + C\e^{Cs}\cD^2 \\
&\hspace{3.5cm} \le C\int_{Q_{3\ep}} |f|^2\e^{2s\vp_2}\dx\dt + C\e^{2s\e^{\la(3\ep + c_0)}}M^2 + C\e^{Cs}\cD^2
\end{align*}
for all $s\ge s_1\ge 1$. Here $M$ and $\cD$ denote a priori bound and measurements defined in Theorem \ref{thm-ISP=0.75}.
Since $\vp_2(x,t)$ attains its maximum at $t = t_0$, we can absorb the first term on the RHS into the LHS by taking $s$ large enough (e.g. $s\ge s_2$). That is 
$$
\int_{\Om_{3\ep}} |f(x)|^2\e^{2s\vp_2(x,t_0)}\dx \le Cse^{2se^{\la(3\ep + c_0)}}M^2 + Ce^{Cs}\cD^2
$$
for all $s\ge s_3 = \max\{s_1,s_2\}$. 
On the hand hand, 
$$
\int_{\Om_{3\ep}} |f(x)|^2e^{2s\vp_2(x,t_0)}\dx \ge \int_{\Om_{4\ep}} |f(x)|^2e^{2s\vp_2(x,t_0)}\dx \ge e^{2se^{\la(4\ep + c_0)}}\|f\|_{L^2(\Om_{4\ep})}^2.
$$
Therefore, we obtain
\begin{equation*}
\|f\|_{L^2(\Om_{4\ep})}^2 \le Cse^{-2\ep_0 s}M^2 + Ce^{Cs}\cD^2
\end{equation*}
for all $s\ge s_3$. Here $\ep_0 := e^{\la(4\ep+c_0)} - e^{\la(3\ep+c_0)} = e^{\la(3\ep+c_0)}(e^{\la\ep} -1) > 0$. 
Since $\sup_{s>0} se^{-\ep_0 s} < \infty$, the above inequality gives
\begin{equation}
\label{ISP:eq6}
\|f\|_{L^2(\Om_{4\ep})}^2 \le Ce^{-\ep_0 s}M^2 + Ce^{Cs}\cD^2
\end{equation}
for all $s\ge \hat{s}$ with $\hat{s} \ge s_3$ satisfying $\hat{s}e^{-\ep_0 \hat{s}}\le C$. 
By substituting $s$ by $s+\hat{s}$, inequality \eqref{ISP:eq6} holds for all $s\ge 0$ with a larger generic constant $Ce^{\hat{s}}$ which is again denoted by $C$. 

Finally, we repeat the argument in the proof of Theorem \ref{thm-LCP=0.75} to show the estimate of H\"{o}lder type:
$$
\|f\|_{L^2(\Om_{4\ep})} \le C(M^{1-\te}\cD^\te + \cD).
$$
The proof of Theorem \ref{thm-ISP=0.75} is completed.
\end{proof}

\section{Conclusions and remarks}
\label{sec-rem}

In this paper, we considered the Carleman estimates for the time-fractional advection-diffusion equation \eqref{equ-gov} and the applications. 

First, in the case of sub-diffusion, that is, the largest fractional order is strictly less than half, the Carleman estimate for the equation \eqref{equ-gov} was established by regarding the fractional order terms as perturbation of the first order time-derivative and the use of the Carleman estimate for the parabolic equations. As an application, the conditional stability for a lateral Cauchy problem was obtained, say, the solution of the equation \eqref{equ-gov} is continuously dependent of not only the partial Cauchy data and the source term but also the initial value. Due to our choice of the weight function $\psi_1(x,t)=d(x)-\be t^{2-2\al_1}$, we do not know whether the estimate is valid without the initial value, and this remains open.
On the other hand, the choice of the new weight function $\psi_1$ is not suitable for the study of the inverse problems. As is well known, for dealing with the inverse problems, the Carleman type estimate derived by $d(x)-\be(t-t_0)^2 + c_0$ ($t_0\in(0,T)$) should be better according to the series of theories in \cite{Y09}.
The inverse problems for the equation \eqref{equ-gov} in the case of $\al_1<\f12$ remain open.
The above arguments cannot deal with the case of sup-diffusion where the largest order is greater than half neither. However, in the case of the largest order is rational number and less than $\f34$, we found  the Carleman estimate with a cut-off function for the equation \eqref{equ-gov2} by using the regular weight function $d(x)-\be(t-t_0)^2 + c_0$ ($t_0\in(0,T)$) can be constructed. Then by an argument similar to that in \cite{Y09}, the conditional stability for the inverse source problem (Problem \ref{prob-ISP=0.75}) was proved as well as the conditional stability for the lateral Cauchy problem. The fractional order $\al=\f34$ is the largest one which one can deal with based on our arguments of Carleman estimate. Moreover, constant $\hat{s}$ and the generic constant $C$ and in Theorem \ref{thm-CE<=0.75} depend on the denominator of fractional order $\al$ which prevents us from extending the order to real number by the density of $\Q$ in $\R$. The case of the general order remains open.

Finally, it should be mentioned that the stability inequalities in {Theorem \ref{thm-LCP=0.75}} immediately gives 
the unique continuation result of \eqref{equ-gov2}, i.e. the solution of the equation \eqref{equ-gov2} must vanish in 
the whole domain $Q$ if its initial value is identically $0$ in $\Om$ and the partial boundary data 
$u|_{\Ga\times (0,T)}$ and $\pa_{\nu} u|_{\Ga\times (0,T)}$ are zero. This principle is called a weak type 
unique continuation for the equation \eqref{equ-gov} since the homogeneous initial value is not essential for the 
unique continuation (UC) (e.g., UC for the elliptic equations or UC for the parabolic equations).
However, we cannot repeat this argument to derive the weak unique continuation of \eqref{equ-gov} because the constant $\ve>0$ depends on the choice of $\Om_0$ in the estimate in Theorem \ref{thm-LCP<0.5}.
We refer to \cite{JLLY16} and \cite{SY11} for other kind of weak type unique continuation where the initial value does not vanish but the homogeneous Dirichlet or Neumann boundary condition is required.

\section*{Acknowledgement}
The first author thanks the Leading Graduate Course for Frontiers of Mathematical Sciences and Physics 
(FMSP, The University of Tokyo). The second author thanks Grant-in-Aid for Research Activity Start-up 
16H06712, JSPS. This work was supported by A3 Foresight Program \lq\lq Modeling and Computation of 
Applied Inverse Problems\rq\rq, Japan Society of the Promotion of Science (JSPS). The second and third 
authors are supported by Grant-in-Aid for Scientific Research (S) 15H05740, JSPS.


\end{document}